\documentclass[11pt]{article}
\usepackage{wrapfig,lipsum}
\usepackage{amsmath,amssymb,graphicx,authblk} 
\usepackage{amsfonts}
\usepackage{empheq}
\usepackage[margin=0.9in]{geometry} 
\usepackage{etoolbox}

\usepackage{enumerate}
\usepackage[toc]{appendix}
\usepackage{hyperref}
\usepackage{fullpage}
\usepackage{tikz}
\usepackage[many]{tcolorbox}
\usepackage{bbm}
\usepackage{comment}
\newtheorem{definition}{Definition}[section]
\newtheorem{theorem}{Theorem}

\newtheorem{remark}{Remark}
\usepackage{comment}
\usepackage{enumitem}
\usepackage[utf8]{inputenc}
\usepackage{xcolor}
\usepackage{graphicx}
\usepackage{algorithm}
\usepackage{algpseudocode}
\usepackage{tikz}
\usepackage[mathlines]{lineno}
\usepackage[capitalise]{cleveref}
\usepackage{multirow}
\usepackage[normalem]{ulem}
\newcommand{\stkout}[1]{\ifmmode\text{\sout{\ensuremath{#1}}}\else\sout{#1}\fi}
\usepackage[makeroom]{cancel}
\usepackage{tabularx,ragged2e}
\newcolumntype{L}[1]{>{\RaggedRight\hsize=#1\hsize}X}
\newcolumntype{C}[1]{>{\Centering\hsize=#1\hsize\hspace{0pt}}X}

\newcommand*\linenomathpatch[1]{%
  \cspreto{#1}{\linenomath}%
  \cspreto{#1*}{\linenomath}%
  \csappto{end#1}{\endlinenomath}%
  \csappto{end#1*}{\endlinenomath}%
}
\newcommand*\linenomathpatchAMS[1]{%
  \cspreto{#1}{\linenomathAMS}%
  \cspreto{#1*}{\linenomathAMS}%
  \csappto{end#1}{\endlinenomath}%
  \csappto{end#1*}{\endlinenomath}%
}

\expandafter\ifx\linenomath\linenomathWithnumbers
  \let\linenomathAMS\linenomathWithnumbers
  \patchcmd\linenomathAMS{\advance\postdisplaypenalty\linenopenalty}{}{}{}
\else
  \let\linenomathAMS\linenomathNonumbers
\fi

\linenomathpatch{equation}
\linenomathpatchAMS{gather}
\linenomathpatchAMS{multline}
\linenomathpatchAMS{align}
\linenomathpatchAMS{alignat}
\linenomathpatchAMS{flalign}
\usepackage{float}
\usepackage{soul}
\makeatletter
\patchcmd{\mmeasure@}{\measuring@true}{
  \measuring@true
  \ifnum-\linenopenaltypar>\interdisplaylinepenalty
    \advance\interdisplaylinepenalty-\linenopenalty
  \fi
  }{}{}
\makeatother

\newcommand{\di}[1]{\,\mathrm{d}#1}

\graphicspath{ {./images/} }
\allowdisplaybreaks

\newenvironment{proof}{\paragraph{Proof:}}{\hfill$\square$}
\allowdisplaybreaks
\def\es{\varepsilon}
\newcommand{\V}{\textcolor{black}}
\newcommand{\AM}{\textcolor{black}}
\newcommand{\SN}{\textcolor{black}}
\newcommand{\ME}{\textcolor{black}}
\begin{document}
	
	
    	\title{Numerical  Exploration of Nonlinear Dispersion Effects via a Strongly Coupled Two-scale System}
\author{Surendra Nepal$^a$\footnote{Corresponding author, email: \texttt{surendra.nepal$@$lnu.se}}\quad  Vishnu Raveendran$^b$\quad  Michael Eden$^c$\quad  Rainey Lyons$^d$\quad   Adrian Muntean$^e$}
\affil[]{$^a$Department of Mathematics, Linnaeus University, Växjö, Sweden}
\affil[]{$^b$Institute for Numerical Simulation, University of Bonn, Germany}
\affil[]{$^c$Faculty of Mathematics, University of Regensburg, Germany}
\affil[]{$^d$Department of Applied Mathematics, University of Colorado Boulder, CO, USA}
\affil[]{$^e$Department of Mathematics and Computer Science, Karlstad University, Sweden}

	\date{} 
	\maketitle

    \begin{abstract}
    The effective, fast transport of matter through porous media is often characterized by complex dispersion effects. To describe in mathematical terms such situations, instead of a simple macroscopic equation (as in the classical Darcy’s law), one may need to consider two-scale boundary-value problems with full coupling between the scales, where the macroscopic transport depends non-linearly on local (i.e. microscopic) drift interactions, which are again influenced by local concentrations.
    Such two-scale problems are computationally very expensive as numerous elliptic partial differential equations (cell problems) have to constantly be recomputed.
    In this work, we investigate such an effective two-scale model involving a suitable nonlinear dispersion term and explore numerically the behavior of its weak solutions.
    \noindent  
    We introduce two distinct numerical schemes dealing with the same non-linear scale-coupling: (i) a Picard-type iteration and (ii) a time discretization decoupling.
    In addition, we propose a precomputing strategy where the calculations of cell problems are pushed into an offline phase.
    Our approach works for both schemes and significantly reduces computation times.  We prove that the proposed precomputing strategy converges to the exact solution.
    Finally, we test our schemes via several numerical experiments that illustrate dispersion effects introduced by specific choices of microstructure and model ingredients.
    \end{abstract}
    
		{\bf Key words}: Nonlinear dispersion; Two-scale systems; Weak solutions; Iterative scheme; FEM approximations;  Numerical simulation.
		\\
		{\bf MSC2020}:  65M60, 47J25, 35M30, 35G55
\maketitle

\section{Introduction}\label{In}
Dispersion refers to the macroscopic spreading of solutes in a porous medium as a result of a combination of microscopic molecular diffusion and drift processes.
Mathematical modeling and computing of dispersion effects in porous media is a complex process that typically involves partial differential equations (PDEs) on multiple temporal and spatial scales (see, e.g., \cite{bear2012phenomenological, musuuza2009extended}).
Quantitative descriptions of dispersion can only be made for simple (regular) porous materials when sufficient scale separation can be assumed to hold when comparing the individual contributions of all the physical and chemical processes involved; see, e.g., \cite{Wood2003RVE} or \cite{raveendran2023homogenization}.
An important aspect of porous media research is concerned with the explicit structure of proposed dispersion tensors tailored for given specific real-world applications.
In practice, the understanding of dispersion usually relies either on detailed microscopic numerical simulations (which can be challenging to perform), or on empirical/semi-analytical expressions (often derived by means of volume averaging arguments). 
Such investigations have been carried out by the porous-media community, 
e.g.~\ME{\cite{guo2015dispersion,ling2018hydrodynamic}} and the references therein.
These works refer particularly to the case of reactive flows in soils with applications to petroleum engineering.
Quite interesting (and not fully understood) dispersion effects appear when fast drift interacts with thin structured media supporting textiles (cf., e.g. \cite{Orlik}) or packaging boards (cf., e.g. \cite{Alamin}).
Another important subject of investigation is solutes penetrating into dense hyperelastic materials; see, e.g., \cite{nepal2021moving,Wilmers}.
Dispersion effects play an important role in such settings: specifically, they are strengthened by the ability of this material to expand locally \ME{\cite{Fiori2025,neff2019modelling}}.
This material property potentially speeds up the penetration of the solute particles through the porous material.
We illustrate numerically this specific scenario in Section \ref{application}, as this has inspired the work reported here.

Effective dispersion models may involve nonlinear two-scale problems where $(a)$ the transport of macroscopic concentrations depends nonlinearly on microscopic drift interactions through an effective dispersion tensor and $(b)$ microscopic dynamics are at the same time influenced by the evolution of the local macroscopic concentrations.
Numerically solving such models leads to huge computational challenges and costs because of their nonlinear nature and strong coupling between scales.
In this work, we numerically investigate a strongly coupled two-scale system with nonlinear dispersion that models particle transport in porous media.
The two scales are coupled in such a way that the macroscopic concentration field influences the microscopic evolution, while the microscopic fields contribute to the macroscopic solution through an upscaled averaged transport coefficient, the dispersion tensor.
To address the computational challenges posed by this nonlinear coupling, we propose two numerical schemes and introduce a precomputing strategy to solve the model efficiently.
In this strategy, the computational burden of computing the effective dispersion tensor is shifted to a pre-processing phase -- often referred to as the offline phase \cite{abdulle2012reduced}.
This allows us to avoid solving the microscopic cell problems at each macroscopic point and at each time step.
Instead, we precompute the dispersion tensor for a representative set of parameters and use a linear interpolation of these precomputed values in an online stage to compute the macroscopic solution.
  We prove that the interpolation error due to the precomputing step can be effectively controlled by refining the parameter step size. 
\subsection{Two-scale dispersion model}\label{problem_formulation}
To set the stage, we introduce two distinct and well separated bounded spatial \V{domains in $\mathbb{R}^2$,} that we refer to as the {\em macroscopic} domain, indicated by $\Omega$, and the {\em microscopic} domain, denoted by $Y$. 
We denote the respective spatial variables by $x\in \Omega$ and $y\in Y$.
In the case where $Y$ has an obstacle \V{that is strictly included in $Y$}, we denote the inner boundary of $Y$ by $\Gamma_N$. 
We fix $T>0$ as the final time of the overall reaction-diffusion-drift process, and by $t\in S:=(0, T)$ we denote the time variable.
See \Cref{Fig:GeometrySchem} for an illustration of this geometric set-up. 
\begin{figure}[h]
    \begin{center}
        \begin{tikzpicture}[scale=0.85]
        \filldraw [gray] plot [smooth cycle] coordinates {(0,1) (1,0) (2,-0.5) (3,0)(5,0) (5,3) (0.5,3) };
\path[thick,->] (3.0,2)  edge [bend left] (8.5,2);
\filldraw[black] (3,2) circle (2pt) node[anchor=east]{\large $x$};
\filldraw[black] (2,1) circle (0pt) node[anchor=west]{\LARGE $\Omega$};
\draw (8.5,0) node [anchor=north] {{\scriptsize }} to (12,0) node [anchor=north] {{\scriptsize }}  to (12,2.5)node [anchor=south] {{\scriptsize }} to (8.5,2.5)node [anchor=south] {{\scriptsize }} to (8.5,0);
			\filldraw [black] plot [smooth cycle] coordinates {(10,1.5) (10,1) (10.8,1.35) (10.8,1.8)(10,2)};
	\draw[<-](10.8, 1.9) to (11.5,2.8) node[anchor=south] {{ $\Gamma_{N}$}};
   \filldraw[black] (10,1.5) circle (0pt) node[anchor=east]{\Large $Y$};
\end{tikzpicture}
\caption{\label{Fig:GeometrySchem} Schematic illustration of a typical two-scale geometry: the macroscopic domain $\Omega$ and the microscopic domain $Y$ with the internal boundary 
$\Gamma_N$.}
    \end{center}
\end{figure}
We are interested in producing suitable numerical schemes to approximate weak solutions to the following strongly coupled two-scale problem:
Find a pair of functions $(u,W)$ (with $W= (w_1,w_2)$) satisfying
\begin{subequations}
\begin{empheq}[left=\text{{\Large $(P)$}}\quad\empheqlbrace\quad]{align}
    \partial_t u +\mathrm{div}(-D^*(W)\nabla  u)&=f &\mbox{in}& \hspace{.2cm}  S\times \Omega, \label{homeq1}\\
    u&=0 &\mbox{on}& \hspace{.3cm}  S\times \partial\Omega, \label{macro_boundary_con}\\
    u(0)&= u_0&\mbox{in}& \hspace{.3cm}  \overline{\Omega}, \label{macro_initial_con}\\[.2cm]
    \mathrm{div}_y\left(-D \nabla_y w_i+G(u)B w_i\right)&=\mathrm{div}_y (D e_i) &\mbox{in}& \hspace{.2cm}  Y,\label{cell3}\\
    \left( -D\nabla_y w_i+BG(u)w_i\right)\cdot n_y &= \left( De_i\right)\cdot n_y &\mbox{on}& \hspace{.2cm} \Gamma_N,\label{cellbn3}\\
    w_i\, &\mbox{ is $Y$--periodic}.&&\label{homeqf}
\end{empheq}
 Here, $u=u(t,x)$ denotes the macroscopic concentration, representing the population density of particles traveling through a porous medium and $w_i=w_i(t,x,y)$ ($i=1,2$) denote the microscopic cell functions. The concept of cell function is best understood in the framework of homogenization theory; see, e.g., \cite{BLP} for a general overview.
The reader who is not familiar with this upscaling technique may simply perceive $w_i(t, x, y)$ as a computable object that encode averaged microscopic information (e.g., the shape and connectivity of the microscopic domain $Y$ and the drift interactions) to the dispersion tensor.
In this work, we consider only the case where $Y, \Omega\subset\mathbb{R}^2$, however, we want to point out that this system can be posed in higher dimensions as well.\footnote{The amount of cell problems is given by the dimensions with $e_i$ being the unit vectors, e.g., in $\mathbb{R}^3$, we would have $W=(w_1,w_2,w_3)$.} Additionally, because the dimensions of the domains do not need to match, the model offers great flexibility.
\AM{It is worth emphasizing that a large variety of  boundary conditions can be defined at the macroscopic boundary $\partial \Omega$ without affecting the solvability of Problem $P$. On the other hand, the conditions imposed across the boundary of the microscopic domain has to ensure the so-called separation of scales, i.e. the homogeneous Neumann and the periodic boundary conditions  are the best options to be imposed here to keep both the solvability of the problem and its micro-macro interpretation.}

\noindent
For the precise context we have in mind, we consider the reaction term   $f\colon S\times\Omega\longrightarrow\mathbb{R}$, the initial condition \(u_0\colon \overline{\Omega}\longrightarrow\mathbb{R}\), the diffusion matrix (at the microscopic scale) $D\colon Y\rightarrow \mathbb{R}^{2\times 2}$, and the drift vector $B\colon Y\rightarrow \mathbb{R}^2$  to be {\em a priori} given functions.
We use the notation $B$ and $D$ for brevity, although both are functions of $y$, i.e., $B=B(y)$ and $D=D(y)$. 
In the derivation and analysis of this two-scale system (\cite{raveendran2023homogenization,raveendran2023strongly}), the drift function $B$ must satisfy certain properties (specified later \SN{in Section \ref{assumption}}) which can be readily met by physically relevant functions.
For example, any solution of a suitable Stokes problem posed on the microscopic domain, $Y$, suffices. 
This is also the approach taken in our numerical assumptions, where the drift $B$ is calculated via a Stokes system, we refer to \cref{appendix}.
This ensures not only that the necessary assumptions are met, but also that the flow respects the microstructure in a physically meaningful way. 

Regarding the specific nonlinear coupling between the microscale and macroscale via the dispersion tensor $D^*$ and the drift interaction $G(u)$:
Owing to \cite{raveendran2023homogenization}, for a given vector
\[
W(t,x,\cdot)=(w_1(t,x,\cdot),w_2(t,x,\cdot))\in H^1(Y)\times H^1(Y),\quad (t,x)\in S\times\Omega,
\]
the effective dispersion tensor $D^*(W)\colon S\times\Omega\to\mathbb{R}^{2\times2}$ arising in \eqref{homeq1} is defined as
\begin{equation}\label{macrodiff}
     D^*(W):=\frac{1}{|Y|}\int_Y D(y)\left(I+\begin{bmatrix}
\frac{\partial w_1}{\partial y_1} & \frac{\partial w_2}{\partial y_1} \\
\frac{\partial w_1}{\partial y_2} & \frac{\partial w_2}{\partial y_2}
\end{bmatrix}\right)\, \di{y},
 \end{equation}
\end{subequations}
where $I\in\mathbb{R}^{2\times2}$ is the identity matrix and $|Y|$ is Lebesgue measure of $Y$.
At first sight, $D^*(W)$ looks like an effective diffusion matrix.
However, with a closer look, we associate \eqref{macrodiff} with dispersion as the entries in $D^*(W)$ depend implicitly on the drift $B$\footnote{Without any drift, \cref{macrodiff} is precisely the effective diffusivity matrix expected by standard homogenization techniques.}.
The nonlinear coupling enters through the term $G(u)$ in \cref{cell3}, where $G\colon\mathbb{R}\to\mathbb{R}$ encodes the interaction between macroscopic concentration and microscopic dynamics. 


In the remainder of the paper, we refer to \eqref{homeq1}--\eqref{macrodiff} as Problem $(P)$.

\ME{\subsection{Strategy and main results}}
\ME{The main difficulty behind Problem $(P)$, from a numerical point of view, is the interplay between micro- and macroscale:
The macroscopic concentration $u$ enters the microscopic cell problems through the scalar (and potentially nonlinear) parameter $G(u)$ and the corresponding cell solutions $W=(w_1,w_2)$ determine the effective dispersion tensor $D^*(W)$.
As a result, approximation errors introduced at the microscale propagate through the effective dispersion tensor to the macroscopic solution.
To handle this nonlinear feedback, we consider two different decoupling strategies:
\begin{itemize}
    \item[$(i)$] A Picard-type iteration in which the microscale problems are solved using the macroscopic solution from the previous iteration; see \cref{scheme1}
    \item[$(ii)$] A semi-implicit time stepping scheme in which the micro--macro coupling is lagged by one time step: the microscopic problems are solved using the macroscopic solution from the previous time step.
    This avoids the additional iteration loop; see \cref{scheme2}.
\end{itemize}
For the first strategy, which was already introduced and analyzed in~\cite{raveendran2023strongly}, we extend the established results by proving the strong convergence of the microscopic iterates
\(W_k\to W\) in $H^1(Y)^2$ almost everywhere in $S\times\Omega$ (\cref{theorem_well_posedness}).
}

\ME{The main focus of this work, however, is the main computational bottleneck: a cell problem must be solved at every macroscopic node and at every time step (see \cref{Fig:GeometrySchem}).\footnote{\ME{Additionally, at every Picard iteration in scheme 1.}}
In our specific problem, we are able to exploit that the cell problems depend on $(t,x)$ only through the scalar quantity $p=G(u(t,x))$.
Using a priori bounds on $u$, we get a compact parameter interval $[-L,L]$ that contains all relevant values of $p$.
During an offline-phase we solve (or \textit{precompute}) the parametrized cell problems for a finite sample $(p_i)_{i}\subset[-L,L]$, store the associated effective tensors $D_i$, and replace the tensor $D^*(W)$ by a piecewise-linear interpolant $D^{\rm int}(p)$ which we evaluate at $p=G(u(t,x))$.}

\ME{Our main error estimate (\cref{Thrm:InterpErrorEstimate}) shows that the interpolation error remains controlled after propagation through the two-scale coupling.
To be more precise, let $u$ denote the solution to Problem $(P)$ and $u^{\rm int}$ the solution obtained using the interpolated dispersion tensor. Then,
\[
\|u-u^{\rm int}\|_{L^2(S; H^1(\Omega))}\le C\delta
\]
with $\delta>0$ denoting the maximal spacing of the parameter sampling $(p_i)_{i}\subset[-L,L]$.
The implementation is carried out using FEniCS \cite{logg2012automated}.
We show in \cref{numerical_experiments}, for both schemes, this precomputing strategy saves considerable computation time without severe loss of accuracy.
The behavior of our approximation schemes is supported by both theoretical results and numerical examples, as we report here.
}

\subsection{A brief discussion on related literature}\label{ssec:literature}

The existence of weak solutions for this model has been established in \cite{raveendran2023strongly}.
The structure of this two-scale system is the result of rigorous homogenization asymptotics performed in recent work \cite{raveendran2023homogenization} where fast nonlinear drifts, leading to a strong dispersion, were investigated.  
This model considered here is a somewhat simpler but structurally similar model in which the dispersion tensor incorporates fewer effects compared to  \cite{raveendran2023homogenization}.
Nevertheless, our mathematical and numerical analyses cover both effective models as they belong to the same general class of two-spatial-scale mathematical problems. 
The precise structure of the nonlinear drift present in the microscopic equation is derived in \cite{CIRILLO2016436}(see also \cite{raveendran21}).
As a direct consequence of the mathematical analysis and asymptotics work reported in \cite{raveendran2023homogenization,raveendran2023strongly}, we have a good picture of the assumptions required for the dispersion tensor to be physically and mathematically meaningful, \V{ these are listed in Section \ref{assumption}.} 
While the aforementioned references show that the model has been studied analytically, we now aim to explore possible avenues for numerical simulations and how these simulations can be used to explore what dispersion effects can naturally occur with such two-scale systems.

\ME{Two-scale numerical models have been developed in a variety of fields.
A prominent example outside porous-media transport.
A prominent example is the \(FE^2\) framework in nonlinear structural mechanics, where a microscopic boundary-value problem is solved at each macroscopic integration point and its averaged response determines the macroscopic constitutive behavior; see, for instance, \cite{Feyel2003}. This leads to computational challenges similar to those considered here, namely the repeated solution of parameter-dependent microscopic problems during the macroscopic simulation.
}
Within porous-media applications, they are utilized to compute the fluid flow in fractured deforming porous media \cite{rethore2007two}, the effect of flow on mineral dissolution in porous medium \cite{bringedal2020phase},
and solid--liquid phase transitions with dendritic microstructure \cite{eck2002two}.
The mathematical literature on the numerical approximation of balance laws posed on two spatial scales has also grown considerably in recent years.
For example, \cite{ray2019numerical, olivares2021two} develop numerical schemes for solving micro–macro models in mineral dissolution and precipitation processes. 
In \cite{ray2019numerical}, the author utilizes the extended finite element method to evaluate cell problems while solving transport equations by applying mixed finite elements. 
In \cite{olivares2021two}, the authors employ a two-scale iterative scheme with an adaptive strategy to improve numerical efficiency at both scales. 



\ME{Different approaches have been proposed and developed to reduce the computational cost of two-scale (or, more generally, multiscale) problems involving repeated microscopic solves.
In the multiscale finite element method \cite{hou1997multiscale}, the microscopic information is incorporated into local basis functions (or local sampling problems in the case of the closely related heterogeneous multiscale method \cite{Abdulle2012-nr}) which allows the global problem to be solved on a coarse mesh.
This framework is particularly well-suited to one-way coupled problems (e.g., when $G(u(t,x))$ is replaced by a known function $g(t,x,y)$) where the local basis functions can be calculated offline.
In nonlinearly coupled problems such as ours, however, the local basis functions become dependent on the macroscopic solution which introduces an additional level of complexity and may substantially increase the online computational cost.
We refer to \cite{Evendiev04}, where such nonlinear couplings have been explored within the framework of the multiscale finite element method.}

\ME{A second class of methods, which encompasses our precomputing approach, relies on database, library, or tabulation strategies.
In these methods, microscopic solutions (or the corresponding effective quantities) are either computed offline for a representative set, as in~\cite{Chen1995-ly,tan2007multiscale}, or generated during the online simulation and stored for subsequent reuse, as in \cite{pope1997computationally,rocha2021onthefly}.
One general drawback of constructing an offline database is that the \textit{realizable} parameter region (i.e., possible parameter values) may be substantially larger than the actual \textit{accessed} region (i.e., attained values during the simulation), which can make the offline phase unnecessarily heavy computationally \cite{pope1997computationally}.
This issue is less pronounced in our case since the coupling can be described via a 1-D admissible parameter region $[-L,L]$ whose bounds are determined via a priori estimates on the macroscopic solution.
This makes this precomputing strategy a good fit for our specific problem.
Closely related precomputing and interpolation strategies have been applied to other multiscale problems in~\cite{le2023numerical,eden2026twoscale}.
}

\ME{Another widely used alternative is the so called \textit{reduced basis} approach, see, e.g., \cite{boyaval2008reduced}, where the cell-solution are approximated in a low-dimensional space constructed from previously computed snapshot solutions.
Similar to our precomputing strategy, this method utilizes an offline--online decomposition: full cell solutions $W_{p_i}$ are computed offline for a finite parameter sample $(p_i)_{i=1}^N$. 
However, instead of directly approximating the nonlinear $D^*(W)$ via the interpolation $D^{\rm int}(p)$, these solutions are used as basis functions to solve $W_p^N$ online over the reduced space $V_N={\rm{span}}\{W_{p_1},\dots,W_{p_N}\}$.
In our setting, direct precomputation of the effective tensor is particularly advantageous as the macroscopic equation only needs $D^*(W_p)$ not the actual cell-functions $W_p$.}

\ME{Finally, we mention adaptive techniques, such as the approach proposed in~\cite{redeker2013fast,redeker2016upscaling}.
There, the microscopic problems for different macroscopic nodes are compared using a suitable similarity criterion and the microscopic problems are solved only at a dynamically selected set of active nodes.
In contrast to reduced-basis and precomputing approaches, this strategy is predominantly online (the similarity criterion is chosen a priori, typically based on offline sensitivity studies).
In our context, this means that $D^*(W_p)$ would only be recalculated if the parameter value $p=G(u(t,x))$ changes significantly (w.r.t.~the similarity criterion).}

\subsection{Outline of the paper}

This paper is organized as follows: The description of the model equations is presented in \cref{problem_formulation}.
In \cref{notations} we list the notation used, the restrictions on data and parameters needed, and the definition of the weak formulation of the problem.
In \cref{numerical_scheme}, we define two numerical schemes able to approximate the weak solution to the model equations and recall any related results.
In \cref{implementation}, we list the pseudocode algorithms behind the implementation of both schemes. We add here a  discussion of the proposed precomputing strategy. 
Finally, in \cref{numerical_experiments}, we test the computational efficiency of the numerical schemes (with and without the precomputing strategy) using practical examples.
Our final comments and a brief discussion of potential future work are included in \cref{conclusion}. 

\section{Notation, assumptions, and concept of weak solution} \label{notations}

In this section, we present our notation for the function spaces used, the basic assumptions on the data, and the definition of the weak formulation of Problem $(P)$.
We also reference a well-posedness result for our two-scale system.

We assume that our macroscopic domain $\Omega\subset\mathbb{R}^2$ is a bounded $C^{2,\alpha}$-domain for some $\alpha\in(0,1)$ and that our microscopic domain $Y\subset\mathbb{R}^2$ is of the form $Y=(0,1)^2\setminus\overline{Y_0}$ where $Y_0\subset\subset(0,1)^2$ is a Lipschitz domain.
We also assume $Y$ to be connected.
Throughout the paper, we make use of the following function spaces: 
\begin{align*}
    \mathcal{U}&:=\{v\in L^2(S;H_0^{1}(\Omega)): \partial_t v\in L^2(S;H^{-1}(\Omega))\},\\
    H_{\#}^1(Y)&:=\{v\in H^1(Y): v \;\mbox{is}\; Y- \mbox{periodic}\},\\
    \mathcal{W}&:=\left\{v\in H_{\#}^1(Y): \int_Yv(y)\di{y}=0\right\},
\end{align*}
where $H_{\#}^1(Y)$ is equipped with standard $H^1(Y)$ norm. We denote $\langle\cdot,\cdot\rangle$ as the duality pairing between $H^{-1}(\Omega)$ and $H_0^1(\Omega).$
We use the standard notation and definitions (Sobolev and Bochner spaces) and refer the reader, e.g., to \cite{adams2003} for the needed details on the corresponding norms, inner products, and eventually other properties.

\subsection{Assumptions and concept of weak solution} \label{assumption}
 The assumptions on the model ingredients are as follows:

 \begin{enumerate}[label=({A}{{\arabic*}})]
 \item  
The microscopic diffusion matrix satisfies $D\in (H^1_\#(Y)\cap L^\infty(Y))^{2\times2}$. Additionally,  there exists $\theta>0$  such that for all  $\eta\in \mathbb{R}^2$ and for almost all  $y\in Y$, it holds
 \begin{equation*}\label{el}
 \theta |\eta |^{2} \leq
D\eta\cdot \eta. 
 \end{equation*}\label{A1}
\item The drift interaction term $G\colon\mathbb{R}\rightarrow\mathbb{R}$ is a locally Lipschitz function, meaning it is Lipschitz continuous over compact sets.
 	 \label{A2}
  \item The microscopic drift velocity  
 $B\in (H^1_\#(Y)\cap L^\infty(Y))^2$ satisfies the following conditions:
\begin{equation*}
\begin{cases}
 	 \mathrm{div}B=0\hspace{.4cm}\mbox{in} \hspace{.4cm}   Y,\\
 	  B\cdot n_y =0 \hspace{.4cm}\mbox{on} \hspace{.4cm}  \Gamma_N.
 	 \end{cases}
 	 \end{equation*}
 \label{A3}
\item
 $f\in C^{\alpha,\frac{\alpha}{2}}(S\times\Omega)$ and $u_0\in C^{2+\alpha}(\Omega)$ \V{with the same $\alpha\in(0,1)$ as in the $C^{2,\alpha}$-regularity assumption on $\Omega$.}
\label{A4}
 \end{enumerate}
The assumptions \ref{A1}--\ref{A4} are taken from our recent work \cite{raveendran2023strongly} where we explain their meaning. 
These assumptions are crucial in proving the existence of solutions to problem $(P)$.
  Having in mind \ref{A1}--\ref{A4}, we can now define the solution to Problem $(P)$ in a weak sense as follows:

\begin{definition}\label{D1}
We say that the pair $(u,W)$ is a weak solution to Problem $(P)$ if $u\in\mathcal{U}$ with $u(0,\cdot) = u_0$ and, for almost every $(t,x)\in S\times \Omega$, $W(t, x, \cdot)\in \mathcal{W}^2$  satisfy
\begin{subequations}
\begin{align}
       \langle \partial_t u, \phi \rangle+\int_{\Omega} D^*(W) \nabla u \cdot\nabla \phi \di{x} &=  \int_{\Omega} f \phi \di{x},\label{weakformmacro}\\
       \int_{Y}\big(D\nabla_y w_i  - G(u(t,x))Bw_i\big)\cdot \nabla_y \psi\di{y}&=\int_{Y} \mathrm{div}_y( D e_i)\psi \di{y}-\int_{\Gamma_N} De_i\cdot n_y\psi \di{\sigma},\label{weakformmicro}
\end{align}
\end{subequations}
for all  $(\phi, \psi)\in H^1(\Omega)\times H_{\#}^1(Y)$ and $i\in\{1,2\}$.
\end{definition}

We clarify the weak solvability of Problem $(P)$ via the following theorem:
\begin{theorem}[Solvability of Problem $(P)$, cf. {\cite[Theorem 2 and Lemma 7]{raveendran2023strongly}}]\label{theorem_existance_of_p}Assume \ref{A1}--\ref{A4} hold.
Then there exists a unique pair 
$$(u,W)\in \mathcal{U}\times  L^\infty(S\times\Omega;\mathcal{W}^2)$$ 
that is a weak solution to the nonlinear parabolic-elliptic system \eqref{homeq1}--\eqref{homeqf} in the sense of \cref{D1}. Furthermore,  we also have that $u, |\nabla u| \in L^\infty(S\times\Omega)$.
\end{theorem}

\section{Two approximation schemes} \label{numerical_scheme}
In this section, we construct two distinct numerical schemes to approximate the solution to the weak form  \eqref{weakformmacro}--\eqref{weakformmicro} of the Problem $(P)$.

\subsection{Finite element approximations for two-scale problems} \label{Iter_discrete}
Before explicitly defining the numerical schemes studied in the rest of the manuscript, we briefly review the general ideas and necessary tools for two-scale Galerkin approximations. 
Following the description in \ME{\cite{ciarlet2002finite}}\SN{\cite{matache2002two, Abdulle2012-nr}}, we let $\mathbb{P}^1$ refer to the space of polynomials of degree $1$ and let  $\mathcal{D}_{H}$ be a finite subdivision of the macroscopic domain, $\Omega$, of triangular elements,  $D \in \mathcal{D}_{H}$, each with diameters $H_D$. 
Then, we define $H:= {\max}_{D \in \mathcal{D}_{H}} H_D$  to be the global macroscopic mesh size. 
We define the finite element space  $\mathbb{V}_H \subset H_0^1(\Omega)$  by
\begin{align*}
 \mathbb{V}_H := \{ v \in \mathcal{C}(\bar{\Omega}) |\;\; v\vert_D \in \mathbb{P}^1(D)\;\; \text{for all}\;\; D \in \mathcal{D}_{H},\; v = 0\;\; \text{on}\;\; \partial \Omega \}.
 \end{align*}

Similarly, we subdivide the microscopic domain, $Y$, with triangular elements $K \in \mathcal{K}_{h}$ and global microscopic mesh size $h:= {\rm max}_{K \in \mathcal{K}_{h}} h_K$, 
where $\mathcal{K}_{h}$ represents a finite subdivision of $Y$ and $h_K$ is the diameter of the element $K \in \mathcal{K}_{h}$.   
We then define the finite element space  $\mathbb{W}_h \subset \mathcal{W}$ as
\begin{align*}
\mathbb{W}_h := \{ \phi \in \mathcal{C}(\bar{Y}) |\;\; \phi \vert_K \in \mathbb{P}^1(K)\;\; \text{for all}\;\; K \in \mathcal{K}_{h}, \phi \;\; \text{is $Y$ periodic}\}.
\end{align*}

Let $\mathcal{N}_1$  and $\mathcal{N}_2$ be the sets of degrees of freedom (DOF) of $\mathcal{D_H}$ and $\mathcal{K}_h$, respectively. 
We then fix a sets of basis functions $\{\xi_l\}_{l \in \mathcal{N}_1}$ and  $\{\eta_j\}_{j \in \mathcal{N}_2}$ such that span$(\xi_l) = \mathbb{V}_H$ and span$(\eta_j) = \mathbb{W}_h$, to approximate the macroscopic and microscopic solutions, $u$ and $w$ via
\begin{align}
\label{macro_solution}
&u^{H}(t, x) := \sum_{l\in \mathcal{N}_1} \alpha_l(t) \xi_l(x),\\
\label{micro_solution}\text{and }&w_{i}^{h} (t, \tilde{x}, y) := \sum_{j\in \mathcal{N}_2} \beta_{i,j}(t, \tilde{x})\eta_j(y), \quad \text{for $i \in \{1,2\}$}.
\end{align}
The term  $\alpha_l(t)$ arising in \eqref{macro_solution} is the macroscopic Galerkin projection coefficient corresponding to $l$th degree of freedom at time $t \in S$.
Likewise, the term  $\beta_{i,j}(t, \tilde{x})$ in \eqref{micro_solution} is the microscopic Galerkin projection coefficient corresponding to the $j$th degree of freedom for some macroscopic discrete node $\tilde{x}$ \SN{of the mesh $\mathcal{D_H}$} at time $t$.
It is important to track the dependence of the macroscopic variables as they are treated as parameters in the microscopic problem \eqref{cell3} through the drift function $G$. 
The problem then turns to decoupling the two scales and making use of the weak form of problem $(P)$ to calculate the approximations \eqref{macro_solution} and \eqref{micro_solution}.
In general, approximating solutions to nonlinear coupled problems posed on multiple space scales is a complex matter as there are multiple ways to proceed and many computational challenges to subvert.
Hence, it is often {\em a priori} unclear which route is most effective. 
In the next two sections, we will propose two distinct numerical schemes which make use of different decoupling techniques along with the weak form of the Problem $(P)$ \eqref{weakformmacro}--\eqref{weakformmicro}.
The first (scheme $1$) refers to a Picard-type iterative scheme (see \cref{scheme1}), while the second (scheme $2$) exploits the decoupling of the problem via a time discretization (see \cref{scheme2}).
As we will see later, while scheme $1$ is effective for mathematical analysis of Problem $(P)$ and convergence results are available \cite{raveendran2023strongly}, it comes with a high computational cost.
On the other hand, scheme $2$ outperforms scheme $1$ without much additional complexity.
Additionally, both schemes can make use of a precomputing step to greatly improve the performance of both schemes (see \cref{precomputing_strategy}).

\subsection{Scheme $1$: linearization by Picard-type iteration} \label{scheme1}
 In this section, we \V{recall a Picard-type iteration method proposed in our recent work \cite{raveendran2023strongly} and then construct its fully discrete finite element scheme, which we call scheme 1; see Definition \ref{Scheme1_fully_def}.} 
\begin{definition}[\V{Picard-type iteration}]\label{D2}
  Let $u^0 \equiv u_0$ and iteratively define the sequence $(u^{k+1}, W^k)_{k\in \mathbb{N}\cup\{0\}}$  as the solution to the following weak form: for almost all $(t,x) \in S\times \Omega$,
\begin{subequations}\label{weakform_iter}
\begin{align}
       \langle \partial_t u^{k+1}, \phi \rangle+\int_{\Omega} D^*(W^{k}) \nabla u^{k+1} \cdot\nabla \phi \di{x} &=  \int_{\Omega} f \phi \di{x}, \label{weakform_itermacro}\\
       \int_{Y}\big(D\nabla_y w_i^k  - G(u^k(t,x))Bw_i^k\big)\cdot \nabla_y \psi\di{y}
       &=\int_{Y} \mathrm{div}_y( D e_i)\psi \di{y}-\int_{\Gamma_N} De_i\cdot n_y\psi \di{\sigma},\label{weakform_itermicro}
\end{align}
\end{subequations}
for all  $(\phi, \psi)\in H^1_{\SN{0}}(\Omega)\times H_{\#}^1(Y)$, $i\in\{1,2\}$, and additionally $u^{k+1}(0,\cdot) = u_0$.
\end{definition}

The central idea behind this \SN{iterative} \V{method is to decouple the two-scale problem by fixing the macroscopic solution from the previous iteration and then solving the corresponding family of microscopic cell problems. 
We recall this iterative method because it gives the analytical basis for scheme 1, the fully discrete formulation introduced below}. \V{The analytical results presented here heavily rely on results in \cite{raveendran2023strongly}.
For instance, there we established} the well-posedness and $k$ independent energy estimates of this iterative scheme. 
Moreover, we also ensured that the sequence of macroscopic solutions, $(u^k)_{k\in \mathbb{N}\cup\{0\}}$, converges to the weak solution $u$ of the original elliptic-parabolic Problem $(P)$. 
However, the convergence of the microscopic sequence, $(W^k)_{k\in \mathbb{N}\cup\{0\}}$, is not established in the previous literature. 
In the following Theorem, we remedy this gap and state the full convergence result. 

\begin{theorem}
\label{theorem_well_posedness}
   Assume \ref{A1}--\ref{A4} hold and set $u^0 = u_0$. Then there exists a sequence 
    \[
    (u^{k+1},W^k)_{k\in \mathbb{N}\cup\{0\}}\subset\mathcal{U}\times L^\infty(S\times\Omega;\mathcal{W}^2)
    \]
    such that for each $k\in\mathbb{N}\cup\{0\} $, the pair $(u^{k+1},W^k)$ uniquely solves  \eqref{weakform_itermacro}--\eqref{weakform_itermicro} in the sense of \cref{D2}, $u^k,|\nabla u^{k}|\in L^{\infty}(S\times\Omega)$ and satisfy
   \begin{subequations}

\begin{align}\label{eq:ukbound}
  \|u^{k}\|_{L^{\infty}(S\times\Omega)}&\leq \|u_0\|_{L^\infty(\Omega)}+T\|f\|_{L^{\infty}(S\times\Omega)},
\end{align}  
     \begin{equation}\label{eq:gradukbound}
        { \|\nabla u^{k}\|_{L^{\infty}(S\times\Omega)}\leq C,}
    \end{equation}
    \end{subequations}
    where $C>0$ independent of $k$.  
    Moreover, we have
    \begin{subequations}
    \begin{alignat}{4}
         u^{k}&\rightarrow u \hspace{1cm} &&\mbox{strongly in} \hspace{.5cm} &&L^2(S\times\Omega)\label{T3e1},\\
         \nabla u^{k}&\rightharpoonup\nabla u \hspace{1cm} &&\mbox{weakly in} \hspace{.5cm} &&L^2(S\times\Omega)\label{T3e2},\\
         \partial_t u^{k}&\rightharpoonup\partial_t u \hspace{1cm} &&\mbox{weakly in} \hspace{.5cm} &&L^2(S;H^{-1}(\Omega))\label{T3e3},\\
         D^*(W^k) &\rightarrow D^*( W )\hspace{1cm} &&\mbox{strongly in} \hspace{.5cm} &&L^2(S\times\Omega)\label{T3e4},
    \end{alignat}
    and for a.e. $(t,x)\in S\times \Omega$ it also holds that  
    \begin{equation}\label{T3e5}
         W^k \rightarrow W  \hspace{1cm}\mbox{strongly in} \hspace{.5cm} (H^1(Y))^2.
    \end{equation}
    \end{subequations}
    Finally, $(u,W)\in \mathcal{U}\times L^\infty(S\times\Omega;\mathcal{W}^2)$ solves Problem (P) in the sense of \cref{D1}.
\end{theorem}
\begin{proof}
\V{For the well-posedness of the iterative problem, the uniform bound \eqref{eq:ukbound}-\eqref{eq:gradukbound}  and the macroscopic convergence \eqref{T3e1}-\eqref{T3e4}, we refer \cite[Theorem 1, Theorem 2, Lemma 3]{raveendran2023strongly}.}
It remains though to prove \eqref{T3e5}. 

To this end, we follow similar techniques as those used in \cite[Lemma 1]{raveendran2023strongly}. 
Let $k\in\mathbb{N}$ be arbitrarily fixed and define $\overline{w}_i^k:=w_i^k-w_i$. 
Taking the test function $\psi:=\overline{w}_i^k$ in both \eqref{weakformmicro} and \eqref{weakform_itermicro} and subtracting the obtained results yields the following expression:
\begin{equation}\label{eq:wk-w}
    \int_{Y}\big(D(y)\nabla_y \overline{w}_i^k- B(y)(G(u^k)w_i^k-G(u)w_i))\big)\cdot \nabla _y \overline{w}_i^k\di{y}=0.
\end{equation}
Adding 
$0$ to equation \eqref{eq:wk-w} in the form of the expression $\int_Y (B(y)G(u)w_i^k- B(y)G(u)w_i^k)\cdot \nabla _y \overline{w}_i^k\di{y}$  and then using \ref{A1}, we have
\begin{equation}\label{eq:add and substract BGw_i}
    \theta\int_{Y} |\nabla_y \overline{w}_i^k |^2 \di{y}
    \leq \int_{Y}B(y)\big((G(u^k)-G(u))w_i^k+G(u)\overline{w}_i^k\big)\cdot \nabla _y \overline{w}_i^k \di{y}.
\end{equation}
Benefiting from the properties of $B$ described in \ref{A2}, the periodicity of $w_i,w_i^k$ and the corresponding integration by parts lead to
\begin{align}
    \int_{Y}B(y)G(u)\overline{w}_i^k\cdot \nabla _y \overline{w}_i^k \di{y}=\frac{1}{2}\int_{Y}B(y)G(u) \cdot \nabla _y (\overline{w}_i^k)^2\di{y}
    =0\label{eq:intb=0}.
\end{align}
Since $G$ is locally Lipschitz and $u^k$ and $u$ are bounded, we have 
\begin{equation}\label{eq:lipschitz count of G}
    |(G(u^k)-G(u))|\leq C|u^k-u|,
\end{equation}
almost everywhere. 
From \cite[(ii) of Lemma 1]{raveendran2023homogenization} and Poincaré-Wirtinger's inequality, there exist $C>0$ independent of $k$, such that
\begin{equation}\label{eq:l2bound of wk}
    \int_Y |w_i^k|^2\di{y}\leq C.
\end{equation}
Using   \eqref{eq:intb=0} and \eqref{eq:lipschitz count of G}  on \eqref{eq:add and substract BGw_i}, we get
\begin{equation}\label{eq:}
    \theta\int_{Y} |\nabla_y \overline{w}_i^k |^2 \di{y}
    \leq C \int_{Y}|u^k-u||w_i^k|| \nabla _y \overline{w}_i^k |\di{y}.
\end{equation}
By Young's inequality in \eqref{eq:}, we obtain
\begin{equation*}
    \theta\int_{Y} |\nabla_y \overline{w}_i^k |^2 \di{y}
    \leq C(\theta)|u^k-u|^2 \int_{Y}|w_i^k|^2 \di{y}+\frac{\theta}{2}\int_{Y}| \nabla _y \overline{w}_i^k |\di{y}.
\end{equation*}
After conveniently rearranging terms and using \eqref{eq:l2bound of wk},  we are led to
\begin{equation}\label{eq:wk convergence}
    \int_{Y} |\nabla_y \overline{w}_i^k |^2 \di{y}
    \leq C |u^k-u|^2.
\end{equation}
Notice that, from \eqref{T3e1}, we get $u^k\rightarrow u$ for almost every $(t,x)\in S\times \Omega.$ Consequently, from \eqref{eq:wk convergence} we also have for almost every $(t, x)\in S\times \Omega$ that the following convergence holds true
\begin{equation}\label{eq:h1convergence}
    \int_{Y} |\nabla_y \overline{w}_i^k |^2 \di{y}\rightarrow 0.
\end{equation}
Since $\int_Y w^i \di{y}=0 $, using \eqref{eq:h1convergence} together with Poincaré-Wirtinger's inequality (see \cite[Chapter 9]{brezis2011functional}), we obtain \eqref{T3e5}.
\end{proof}

With this convergence result, we can then make use of standard finite element methods on the iteration problem \eqref{weakform_iter}\SN{, which yields a numerical scheme, referred to as scheme 1}.
We begin by discussing the discretization of the space and time domains.
To deal with the discretization in time, we decompose $\bar S= [0, T]$ into $M$ sub-intervals. 
Let $\Delta t := T/M$ be the uniform time step size with $t_n := n \Delta t$ for $n\in \{0, \cdots, M\}$. We use  $u_{n-1}^{k, H}$ to denote the fully discrete $k$th iterative solution of $u$ at time $t = t_{n-1}$. 
Similarly,  $w_{i, n-1}^{k, h}(\tilde{x}, \cdot)$ represents  the corresponding iterative solution of $w$ at time $t = t_{n-1}$ and macroscopic \SN{discrete} node $\tilde{x}$ \SN{of $\mathcal{D}_H$}.
Once we solve for $w_{i, n-1}^{k,h}(\tilde{x}, \cdot)$  for all nodes $\tilde{x}$, we compute estimations of the dispersion tensor via $D^{*}(w_{1, n-1}^{k,h}, w_{2, n-1}^{k,h})$.
By utilizing an implicit Euler method to approximate the time derivative,  the fully discrete weak formulation for the macroscopic equations reads as follows:
\begin{equation}
\label{weak_parabolic_scheme1}	\displaystyle \int_{\Omega}\frac{u_{n}^{k+1, H}-u_{n-1}^{k+1, H}}{\Delta t} \Psi \di{x} +  \int_{\Omega}D^{*}(w_{1, n-1}^{k,h}, w_{2, n-1}^{k,h})  \nabla u^{k+1, H}_n\cdot\nabla \Psi \di{x}= \int_{\Omega}f_{n} \Psi \di{x},
\end{equation}
 for all $\Psi \in \mathbb{V}_H$.

Given $u_{n-1}^{k+1, H}$,  we can now compute $u_{n}^{k+1, H}$ by solving the following macroscopic equations
\begin{equation}
\label{S3}	\displaystyle \int_{\Omega}u_{n}^{k+1, H} \Psi \di{x} + \Delta t \int_{\Omega}D^{*}(w_{1, n-1}^{k,h}, w_{2, n-1}^{k,h}) \nabla u^{k+1, H}_n\cdot\nabla \Psi \di{x}=  \int_{\Omega}(\Delta tf_{n} + u_{n-1}^{k+1, H}) \Psi \di{x},
\end{equation}
 for all $\Psi \in \mathbb{V}_H$.
The fully discrete weak formulation of our original problem treated with \SN{Picard-type iteration} is as follows: 

\begin{definition} [Scheme 1]
  Given $u^{k, H}_n, u^{k+1, H}_{n-1}\in \mathbb{V}_H$ with $u_0^{k+1, H} = u_0^{k, H} = \tilde{u}$. Here $\tilde{u}$ denotes the projection of the given initial condition $u_0$ in $\mathbb{V}_H$.  We define  the pair 
  $$(W^{k, h}_{n-1} (\tilde{x}, \cdot), u^{k+1, H}_n)\in \mathbb{W}_h^2 \times \mathbb{V}_H$$ to be the fully discrete weak solution to \eqref{weakformmacro}--\eqref{weakformmicro} \SN{after the Picard-type iteration}  if for all $v\in \mathbb{W}_h$,  $\Psi \in \mathbb{V}_H$ and for all $n \in \{1, \cdots, M \}$ the following identities hold:
 \begin{subequations} \label{Scheme1_fully_def}
\begin{align}
 \nonumber \int_Y \left(D(y)\nabla_y w_{i, n-1}^{k, h} (\tilde{x}, y)-  G(u^{k, H}_{n-1}(\tilde{x}))  B(y) w_{i, n-1}^{k, h} (\tilde{x}, y)  \right)&\cdot \nabla_y v(y) \di{y} \\
 \label{scheme1_micro_fully_def}  +  \int_{\Gamma_N} D(y)e_i\cdot n_y v \di{\sigma} & =  \int_Y  \nabla_y D(y) \cdot  e_i v \di{y},  \\
  \label{scheme1_macro_fully_def}  
  \displaystyle \int_{\Omega}u_{n}^{k+1, H} \Psi \di{x} + \Delta t \int_{\Omega}D^{*}(w_{1, n-1}^{k,h}, w_{2, n-1}^{k,h}) \nabla u^{k+1, H}_n\cdot\nabla \Psi \di{x} &=  \int_{\Omega}(\Delta tf_{n} + u_{n-1}^{k+1, H}) \Psi \di{x}.
\end{align} 
\end{subequations}
This scheme is referred to as \emph{Scheme~1} throughout the paper.
\end{definition}

\subsection{Scheme $2$: linearization by time stepping} \label{scheme2}

While analytical results for the \SN{Picard-type iteration} are readily available, the decoupling through iteration turns out to be, as we will see in Section \ref{numerical_experiments}, computationally taxing due to the superfluous iteration step. 
As an alternative, we construct a more natural decoupling via time-stepping to approximate weak solutions to Problem $(P)$. 
Using the same notation as in the previous two sections, given  $u_{n-1}^{H}$,  we first compute the approximation $w_{i, n-1}^{ h}(\tilde{x}, y)$ and then we compute $u_{n}^{H}$. 
We can then continue this calculation for every time step. 
More precisely,  the fully discrete weak formulation of the problem treated with scheme $2$ is as follows: 

\begin{definition} [Scheme 2]
 Given $u_{n-1}^H \in \mathbb{V}_H$ with $u_0^{H} = \tilde{u}$.  We define the pair  $$(W_{n-1}^h, u_n^H)\in \mathbb{W}_h^2 \times \mathbb{V}_H$$ to be the fully discrete weak solution to \eqref{weakformmacro}--\eqref{weakformmicro} with scheme $2$ if for all $v\in \mathbb{W}_h$,  $\Psi \in \mathbb{V}_H$ and for all $n \in \{ 1, \cdots, M \}$ the following  identities hold:
\begin{subequations}\label{Scheme2_fully}
\begin{align}
 \nonumber \int_Y\left(D(y)\nabla_y w_{i, n-1}^{h} (\tilde{x}, y)- G(u^{H}_{n-1}(\tilde{x})) B(y) w_{i, n-1}^{h} (\tilde{x}, y)\right)  \cdot & \nabla_y v(y) \di{y} \\
 \label{scheme2_micro_fully_def} +  \int_{\Gamma_N} D(y)e_i\cdot n_y v \di{\sigma}  &=  \int_Y  \nabla_y D(y) \cdot  e_i v \di{y},  \\
  \label{scheme2_macro_fully_def} 
  \int_{\Omega}u_{n}^{H} \Psi \di{x} + \Delta t \int_{\Omega}D^{*}(w_{1, n-1}^{h}, w_{2, n-1}^{h}) \nabla u^{H}_n\cdot\nabla \Psi \di{x}&=  \int_{\Omega}(\Delta tf_{n} + u_{n-1}^{H}) \Psi \di{x}.    
\end{align} 
\end{subequations} 

\end{definition}

\SN{This scheme is referred to as \emph{Scheme~2}. We refer the reader to Theorem 2 in \cite{olivares2021two}, where the convergence of the time-discrete problem is established. This result provides a suitable framework for establishing the well-posedness results and convergence of Schemes 1 and 2. A complete mathematical analysis is beyond the scope of the present work and therefore postponed to follow-up work.}


\section{Implementation} \label{implementation}
The goal of this section is to present the implementation of the finite element method outlined in \cref{Iter_discrete} in conjunction with the numerical schemes discussed in \cref{scheme1} and \cref{scheme2}.  
For our implementation, we opt to use the FEniCS platform.
We refer the interested reader, for instance, to \cite{logg2012automated} for detailed information on solving partial differential equations using FEniCS.
Initially, we provide a brief overview of our implementation strategy where we provide algorithms for each scheme. 
 We then discuss the precomputing strategy which we will later show considerably improves the computing time for both schemes.  

 \subsection{Overview of implementation}
\SN{The implementation is based on a two-scale finite element discretiztion as described in Section \ref{Iter_discrete}. We use the same microscopic mesh $\mathcal{K}_h$ and finite element space $\mathbb{W}_h$ for all macroscopic discrete node $\tilde{x} \in \mathcal{D}_H$. The dependence on the macroscopic position enters only through the parameter $G(u(\tilde{x}))$ appearing in the microscopic cell problem. Consequently,  for each macroscopic discrete node $\tilde{x}$, we solve a cell problem on the same reference microscopic mesh, but with different coefficients determined by the local macroscopic solution. For simplicity, both the macroscopic and microscopic finite element spaces are chosen as the $P_1$ Lagrange finite element space. It is worth mentioning that the proposed methodology is not restricted to $P_1$ elements; higher-order finite element spaces could be employed as well, although they are not considered in the present work.}\\
  We define a microscopic system by assigning a microscopic grid for each degree of freedom (i.e., node) on the macroscopic grid. 
  In our finite element framework, we solve the system by decoupling the microscopic and macroscopic formulation as described in schemes. In scheme $1$, we use the macroscopic solutions from the previous iteration to solve the microscopic elliptic problem and compute the dispersion tensor for the next iteration. Utilizing the computed dispersion tensor, we then solve the parabolic problem to get a macroscopic solution for the next iteration. 
  The iteration continues until the difference (measured in the $L^2(S; L^2(\Omega))$ norm) between consecutive iterations reaches a predefined tolerance value, $\epsilon$. 
  We describe this computation process in \cref{alg:iterative}, which was originally presented in \cite{raveendran2023strongly}.    
\begin{algorithm}[h!]
\caption{Procedure to approximate the weak solution to Problem $(P)$ by scheme $1$.} \label{alg:iterative}
\begin{algorithmic}[1]
\State \texttt{Discretize the space microscopic domain $Y$ and macroscopic domain $\Omega$}
\State \texttt{Discretize the time domain $[0, T]$ with step size $\Delta t$}
\State \texttt{Solve the Stokes problem \eqref{stoke1}-\eqref{stoke4} to get $B(y)$}
\State \texttt{Choose data $f, D, G, u_0$}
\State \texttt{Set initial iteration $u^0=u_0$}
\State \texttt{Set tolerance value $\epsilon$}
\State \texttt{Set the maximum number of iterations,  Maxiter.}
\State \texttt{Initialize iteration and time. i.e. $iter =0$, $t=0$}
\While{$ iter < Maxiter$}
\State \texttt{Set $u_{\text{old}}=u_0$}
\State \texttt{Set $u = [u_{\text{old}}]$}
\For{\texttt{each time discrete node  on time domain}}
\For{\texttt{each node on macroscopic grid}}
        \State \texttt{Solve for $(w_1, w_2)$ using $G(u_{\text{old}})$}
      \EndFor
      \State \texttt{Compute  $D^*$ from $(w_1, w_2)$}
       \State \texttt{Solve for $u_{\text{new}}$ using $D^*$}
        \State \texttt{append $u$ with $u_{\text{new}}$}
       \State \texttt{$u_{\text{old}} \leftarrow u_{\text{new}}$}
    \EndFor
      \If{$\|u-u^0\|< \epsilon$}
    \State \texttt{Stop}
\EndIf 
\State{$u^0\leftarrow u$}
\EndWhile
\end{algorithmic}
\end{algorithm}

In scheme $2$, we utilize the given macroscopic solutions from the previous time step $n-1$ to solve the microscopic elliptic problem and compute the dispersion tensor. Employing the computed dispersion tensor, we then proceed to solve the parabolic problem to get a macroscopic solution for the next time step $n$.  We continue this time-stepping process until we achieve the final time $T$. The described approach is summarized in Algorithm \ref{alg:fully_discrete}.

  \begin{algorithm}[h!]
\caption{Procedure to  approximate the weak solution to Problem $(P)$ by scheme $2$.} \label{alg:fully_discrete}
\begin{algorithmic}[1]
\State \texttt{Discretize the space microscopic domain $Y$ and macroscopic domain $\Omega$}
\State \texttt{Discretize the time domain $[0, T]$ with step size $\Delta t$}
\State \texttt{Solve the Stokes problem \eqref{stoke1}-\eqref{stoke4} to get $B(y)$}
\State \texttt{Choose data $f, D, G, u_0$}
\State \texttt{Set $u_{\text{old}}=u_0$}
\For{\texttt{each time discrete node  on time domain}}
\For{\texttt{each node on macroscopic grid}}
        \State \texttt{Solve for $(w_1, w_2)$ using $G(u_{\text{old}})$}
      \EndFor
      \State \texttt{Compute  $D^*$ from $(w_1, w_2)$}
       \State \texttt{Solve for $u_{\text{new}}$ using $D^*$}
       \State \texttt{$u_{\text{old}} \leftarrow u_{\text{new}}$}
    \EndFor
\end{algorithmic}
\end{algorithm}

\subsection{Precomputing strategy}\label{precomputing_strategy}
As discussed in the previous section, the microscopic elliptic problem must be solved many times for each macroscopic node.
For both schemes, this process can be quite costly as superfluous calculations are not avoided.
Between these two schemes, the iterative scheme is even more expensive due to the additional loop due to the iteration step. 
However, the microscopic cell problems \eqref{cell3}-\eqref{homeqf} only see the macroscopic solution, $u$, point-wise and can therefore be essentially treated as a parameter.
This leads us to the a``precomputing strategy" which involves solving the microscopic problems \eqref{cell3}-\eqref{homeqf} for a range of macroscopic values \textit{a priori}. 

\begin{subequations}
To make the idea of the precomputing strategy rigorous, let us consider the following auxiliary equations 
\begin{align}
    \mathrm{div}_y\left(- D \nabla_y w_{i,p}+pB w_{i,p}\right)&=\mathrm{div}_y(D e_i) &\mbox{in}& \hspace{.2cm}  Y,\label{auxiliary1}\\
    \left( -D\nabla_y w_{i,p}+pB w_{i,p}\right)\cdot n_y &= \left( De_i\right)\cdot n_y &\mbox{on}& \hspace{.2cm} \Gamma_N,\label{auxiliary2}\\
    w_i\, &\mbox{ is $Y$--periodic},&&\label{auxiliary3}
    \end{align}
\end{subequations}
where $i\in\{1,2\}$ and $p\in [-L, L] \subset \mathbb{R}$. 
Relating back to Problem $(P)$, we chose $L$ as 
\begin{equation*}
    L=\max_{r\in[-m,m]} |G(r)|,
\end{equation*}
with
\begin{equation*}
m:=\|u_0\|_{L^\infty(\Omega)}+T\|f\|_{L^{\infty}(S\times\Omega)}.
\end{equation*}
This choice for $L$ is made based on the inequality $\|u\|_{L^\infty(S\times\Omega)}\leq\|u_0\|_{L^\infty(\Omega)}+T\|f\|_{L^{\infty}(S\times\Omega)}$ proven in \cite[Theorem 1]{raveendran2023strongly}.
We also define the matrix-valued function $\overline{D} \colon [-L, L]\rightarrow M_{2\times2}$, such that
\begin{equation}\label{Eq:DBar}
    \overline{D}(p):=D^*(W_p),
\end{equation}
where $W_p:=(w_{1,p},w_{2,p})$.

Then, the precomputing strategy involves solving system \eqref{auxiliary1}-\eqref{auxiliary3} for a fixed number of discrete values of $p$ and interpolating the entries in the dispersion matrix \eqref{Eq:DBar}.
In other words,  the precomputing strategy reduces to the following steps:
\begin{enumerate}[align=left, label=  Step \arabic*:]
 \item \label{Step1}  Discretize the domain $[-L, L]$ by introducing  a finite set of points given by  $-L\leq p_1<p_2<\cdots <p_r\leq L$ with 
\begin{equation}\label{delta}
    \delta:=\max_{k\in \{1,\cdots, r-1\}}(p_{k+1}-p_{k}).
\end{equation} 
Observe that we do not set any {\em a priori} rules regarding how many grid points to choose. However, adding more points naturally increases the computational cost while improving in the same time the quality of approximation, compare Theorem \ref{Thrm:InterpErrorEstimate}.
\item  Solve the auxiliary problem on the microscopic domain for $p = p_k, k \in\{1,2, \cdots, r\}$ and compute each component of the dispersion tensor $\bar{D}(p_k)$ for $k \in \{1, 2, \cdots, r\}$. It is important to note that this step is fully parallelizable, as the computations for different values of $p_k$ are completely independent and can be performed simultaneously. \label{Step2}
\item  Utilize $\bar{D}(p_k)$ to construct an interpolated dispersion tensor $D^{\rm int}(p)$ for $p \in [-L, L]$.   \label{Step3}
\end{enumerate}
The procedure \ref{Step1}--\ref{Step3} is usually done only once as a preprocessing phase; see Figure \ref{Fig:Diffusion_tensor} for a visualization. It is worth mentioning that we only require \textit{a priori} macroscopic information. 
Now, for each time step and iteration, we use $D^{\rm int}(p)$ to evaluate the dispersion tensor when updating the macroscopic solution for the next time step. 

In our simulation, we use the linearly interpolated values inside the domain of $p \in [-L, L]$ and a constant extrapolation as boundary values if $D^{\rm int}(p)$ needs to be computed outside the domain $[-L, L]$. 
We now show, under additional assumptions, that the difference between the solution of problem $(P)$ and the solution calculated by the precomputing strategy is comparable to the value of $\delta$. 
  
 \begin{theorem}\label{Thrm:InterpErrorEstimate}
     Assume \ref{A1}--\ref{A4} hold. Let $(u,W)\in \mathcal{U}\times  L^\infty(S\times\Omega;\mathcal{W}^2)$ 
 be a weak solution to the Problem $(P)$ and $(u^{\rm int},W^{\rm int})\in \mathcal{U}\times  L^\infty(S\times\Omega;\mathcal{W}^2)$ be a weak solution for the problem 
\begin{align}
      \langle\partial_t u^{\rm int}, \phi \rangle+\int_{\Omega} D^{\rm int}(G(u^{\rm int})) \nabla u^{\rm int} \cdot\nabla \phi \di{x} &=  \int_{\Omega} f \phi \di{x},\label{eq:interpolated solution}\\
     \int_{Y}\big(D\nabla_y w_i^{\rm int}  - G(u^{\rm int}(t,x))Bw_i\big)\cdot \nabla_y \psi\di{y}&=\int_{Y} \mathrm{div}_y( D e_i)\psi \di{y}-\int_{\Gamma_N} De_i\cdot n_y\psi \di{\sigma}\nonumber\\
     &\hspace{2cm} \mbox{for a.e. }(t,x)\in S\times \Omega,
\end{align} 
with $u^{\rm int}(0)=u_0$, $i\in\{1,2\}$ and for every $(\phi, \psi)\in H^1(\Omega)\times H_{\#}^1(Y)$. Assume $\nabla u^{\rm int}\in L^{\infty}(S\times \Omega)$, then, we have
\begin{align}
    \|u-u^{\rm int}\|_{L^2(S\times\Omega)}&\leq  C\delta,\label{eq:interpolation l2 error}\\
    \|\nabla(u-u^{\rm int})\|_{L^2(S\times\Omega)}&\leq  C\delta,\label{eq:interpolation h1 error}
\end{align} 
where $\delta$ is given by \eqref{delta}.
 \end{theorem}
\begin{proof}
    We subtract \eqref{eq:interpolated solution} from the weak formulation of Problem $(P)$  and choose $\phi=u-u^{\rm int}$, we get
    \begin{equation*}
         \langle\partial_t (u-u^{\rm int}), u-u^{\rm int} \rangle+\int_{\Omega} (D^*(W) \nabla u- D^{\rm int}(G(u^{\rm int})) \nabla u^{\rm int})\cdot\nabla (u-u^{\rm int}) \di{x} =  0.
    \end{equation*}   
    Using the uniform positivity of $D^*(W)$ (see \cite[Lemma 3]{raveendran2023strongly}),  we have
    \begin{equation}\label{eq:add and subtract dstar}
        \frac{1}{2}\frac{\mathrm d}{\mathrm dt}\|u-u^{\rm int}\|_{L^2(\Omega)}^2+\vartheta \|\nabla(u-u^{\rm int})\|_{L^2(\Omega)}^2\leq \int_\Omega (D^*(W)-D^{\rm int}(G(u^{\rm int})))\nabla u^{\rm int} \cdot\nabla (u-u^{\rm int}) \di{x},
    \end{equation}
    for some $\vartheta >0.$
    From the definition of $\overline{D}$, 
    \begin{align}
       D^*(W)=\overline{D}(G(u)).
    \end{align}
   Recalling \cite[Section 3.2, Lemma 2 and Lemma 3]{raveendran2023strongly}, we have that $\overline{D}(\cdot)$ is Lipschitz, and hence,  together with assumption \ref{A2}, we get
    \begin{align*}
        |\overline{D}(G(u))-\overline{D}(G(u^{\rm int}))|
        \leq C |u-u^{\rm int}|.
    \end{align*}
    Since $D^{\rm int}(\cdot)$ is the linear interpolation of $\overline{D}(\cdot)$, and $\overline{D}$ is Lipschitz continuous, using triangle inequality one can show 
    \begin{align}\label{eq:interpolation_error}
        |\overline{D}(G(u^{\rm int}))-D^{\rm int}(G(u^{\rm int}))|\leq C \delta,
    \end{align}  
    for some $C>0.$
    Since we assumed $\nabla u^{\rm int}\in L^{\infty}(\Omega)$, adding and subtracting $\overline{D}(G(u^{\rm int}))$ to the right-hand side of \eqref{eq:add and subtract dstar} and using the aforementioned estimates and triangle inequality, yields
    \begin{align}
        \int_\Omega &(D^*(W)-D^{\rm int}(G(u^{\rm int})))\nabla u^{\rm int} \cdot\nabla (u-u^{\rm int}) \di{x}\nonumber \\
        &\leq \|\nabla u^{\rm int}\|_{L^\infty(\Omega)} \int_\Omega\left( |D^*(W)-\overline{D}(G(u^{\rm int}))|+|\overline{D}(G(u^{\rm int}))-D^{\rm int}(G(u^{\rm int}))|\right)|\nabla (u-u^{\rm int})| \di{x}\nonumber\\
        &\leq C \int_\Omega| u-u^{\rm int}||\nabla (u-u^{\rm int})| \di{x}+  \delta C \int_\Omega|\nabla (u-u^{\rm int})| \di{x}.\label{eq:dstar estimatte}
    \end{align}
    Now, with the help of Young's inequality and \eqref{eq:dstar estimatte}, from \eqref{eq:add and subtract dstar}, we obtain
    \begin{equation}\label{eq:1}
        \frac{\mathrm d}{\mathrm dt}\|u-u^{\rm int}\|_{L^2(\Omega)}^2 +{\vartheta}\|\nabla(u-u^{\rm int})\|_{L^2(\Omega)}^2\leq C\|u-u^{\rm int}\|_{L^2(\Omega)}^2+C\delta^2.
    \end{equation}
   So, we have
    \begin{equation}\label{eq:gronwall estimate}
        \frac{\mathrm d}{\mathrm dt}\|u-u^{\rm int}\|_{L^2(\Omega)}^2 \leq C\|u-u^{\rm int}\|_{L^2(\Omega)}^2+C\delta^2.
    \end{equation}
    As an application of Grönwall's inequality on \eqref{eq:gronwall estimate}, we obtain the wanted estimate \eqref{eq:interpolation l2 error}.
    Now, using \eqref{eq:interpolation l2 error} on \eqref{eq:1}, we get \eqref{eq:interpolation h1 error}. 
\end{proof}

\ME{
\begin{remark}\label{rem:InterpOC}
\begin{itemize}
    \item \Cref{Thrm:InterpErrorEstimate} remains valid if the piecewise-linear interpolant \(D^{\rm int}\) is replaced by a piecewise-constant approximation \(D_{\rm pc}^{\rm int}\). Indeed, if \(\overline D\) is Lipschitz continuous and \(\delta\) denotes the maximal spacing of the parameter grid, then
    \[
    \bigl|\overline D(p)-D_{\rm pc}^{\rm int}(p)\bigr|
    \leq
    \delta\,\operatorname{Lip}(\overline D).
    \]
    \item The error rate improves for piecewise-linear interpolation if the parameter-to-tensor map is more regular.
    More precisely, if \(\overline D\in C^{1,1}([-L,L])^{2\times2}\), then 
    \[
    \|\overline D-D^{\rm int}\|_{L^\infty([-L,L])} \leq C\delta^2.
    \] 
    Since stability estimate used in the proof of \Cref{Thrm:InterpErrorEstimate} depends linearly on the coefficient error, this leads to the improved estimate
    \[
    \|u-u^{\rm int}\|_{L^2(S;H^1(\Omega))} \leq C\delta^2.
    \]
    In our case, such a result would require corresponding second-order regularity of the map \(p\mapsto \overline D(p)\).
    We refer to \cite{eden2026twoscale} where a similar result has been shown for a related problem.
    We point out that this improved convergence rate can also be seen in our simulations, see~\cref{tab:precompute_delta}.
    \item 
    More generally, if \(\overline D\in C^{m,1}([-L,L])^{2\times2}\) and \(D_m^{\rm int}\) is a piecewise-polynomial interpolant of degree \(m\), then one has
    \[
        \|\overline D-D_m^{\rm int}\|_{L^\infty([-L,L])}
        \leq C\delta^{m+1}.
    \]
    Since the macroscopic solution error is bounded linearly in terms of the \(L^\infty\)-error of the effective coefficient, this order is inherited by the macroscopic approximation.
\end{itemize}  
\end{remark}
}

 \section{Numerical experiments}\label{numerical_experiments}
 
 In this section, we present simulation results for a test problem, with  
all computations carried out using FEniCS \cite{logg2012automated}. The computations were performed on a standard laptop, employing sparse LU decomposition to solve the linear systems of equations. 
In what follows, we establish various sets of model ingredients and microstructures for the numerical experiments.
Our selection mainly aims to evaluate the capacity of the numerical method while also ensuring interesting simulations.
Although some choices lack physical or mathematical motivation, we highlight any exceptions. We begin by presenting the simulation results for scheme 1 and scheme 2, both with and without precomputing. Next, we compare the computing times for each scenario and perform error analysis to determine the order of convergence \SN{(OC)}. Finally, we apply this model and simulation technique to capture the penetration of liquid into materials.

\subsection{Simulation results with and without precomputing} \label{microsolution}
 We begin by first defining the microscopic domain  
 \begin{equation}
 \label{micro_domain1} Y:= (0,1)^2 \setminus  (\mathcal{E}_{0.1, 0.2}((0.85, 0.75)) \cup \mathcal{E}_{0.3, 0.08}((0.35, 0.1)) \cup \mathcal{E}_{0.15, 0.15}((0.175, 0.8))),\end{equation} 
 where  $\mathcal{E}_{r_1, r_2}((y_1, y_2))$ denotes the closed ellipse with center $(y_1, y_2)$,  semi major axis $r_1$ and semi minor axis $r_2$. 
 Note that in the case $r_1 = r_2$, $\mathcal{E}_{r_1, r_2}((y_1, y_2))$ represents a closed disk of radius $r_1$.
To ensure the microscopic drift $B(y)$ is interesting, physically relevant, and satisfies assumption \ref{A3}, we solve the Stokes problem in each geometry with  the choice of  viscosity  $\mu=0.01$ and the force 
\[F:Y\to\mathbb{R}^2 \mbox{ given by }F(y) = (10\sin(2 \pi y_1) \sin(2\pi y_2), 10\sin(2 \pi y_1) \cos(2\pi y_2)),\]
for $y=(y_1,y_2)\in Y$. 
 We refer the reader to the \cref{appendix} for more details on the weak formulation and solving techniques of the Stokes problem.
 The velocity vector $B(y) = (B_{1}(y), B_{2}(y))$, i.e., the solution to the Stokes problem, is shown in Figure \ref{Fig:Stoke}.

\begin{figure}[ht]
		\centering
		\includegraphics[width=0.50\textwidth]{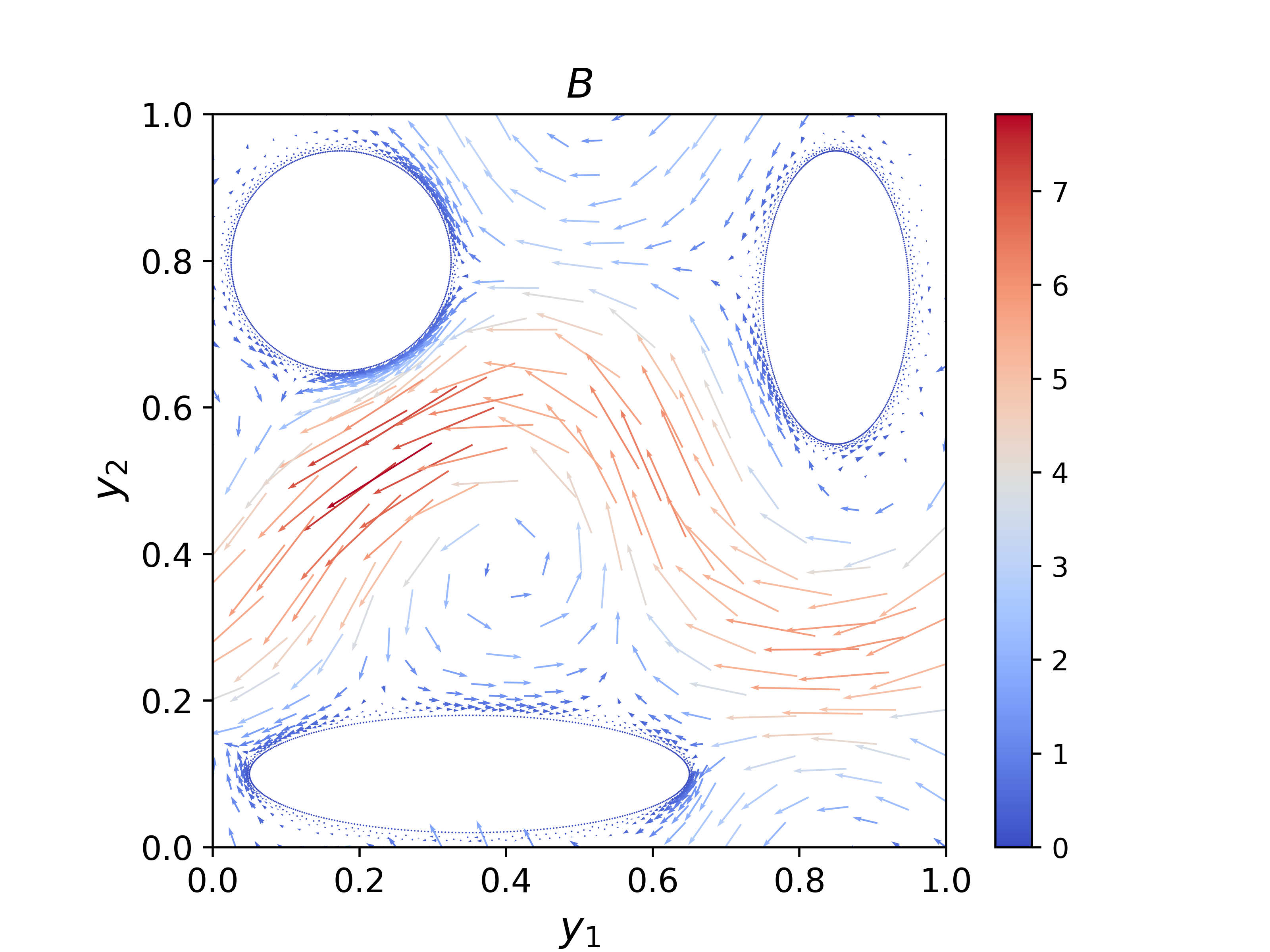}
		\caption{\SN{Velocity field $B(y)$ to the Stokes problem over domain \eqref{micro_domain1}.}}
		\label{Fig:Stoke}
	\end{figure}

 \par While implementing the algorithm that solves the microscopic problem under the zero-average condition,  we employ the Lagrange multiplier method for both schemes as in, for instance,  \cite{formaggia2002numerical}. 
 \SN{To enforce the periodic boundary conditions, the microscopic mesh is chosen to be compatible with the periodic structure of the unit cell $Y$, so that the mesh nodes on opposite boundaries match.
 In FEniCS, the periodicity is imposed through a boundary mapping that identifies the corresponding degrees of freedom on opposite boundaries.
 Consequently, the discrete functions in $W_h$ satisfy the prescribed periodic boundary conditions; see, for example, \cite[p.~413]{logg2012automated} for further details.}
The weak formulation for the microscopic equation \eqref{scheme1_micro_fully_def} with scheme $1$ is then equivalent to 
	solving for $(w_{i, n-1}^{k, h} (\tilde{x}, y), c)\in \mathbb{W}_h \times \mathbb{R} $ such that the following  identity holds:
\begin{align}
  \nonumber &\int_Y\left( D(y)\nabla_y w_{i, n-1}^{k, h} (\tilde{x}, y) - G(u^{k, H}_{n-1}(\tilde{x})) B(y) w_{i, n-1}^{k, h} (\tilde{x}, y)\right)  \cdot \nabla_y v(y) \di{y} \\
 \label{microweak_with_lagrange}  & \hspace{2cm} +  \int_{\Gamma_N} D(y)e_i\cdot n_y v \di{\sigma} + \int_Y d w_{i,n-1}^{k,h}(y) \di{y} + \int_Y c v \di{y}    =  \int_Y  \nabla_y D(y) \cdot  e_i v \di{y},  
 \end{align}
 for all $ (v, d) \in \mathbb{W}_h \times \mathbb{R}, \;\;i \in \{1, 2\}$.

Similarly, using the method of Lagrange multipliers, the weak formulation for the microscopic equation \eqref{scheme2_micro_fully_def} with scheme $2$ is  equivalent to 
	finding $(w_{i, n-1}^{h} (\tilde{x}, y), c)\in \mathbb{W}_h \times \mathbb{R} $ such that the following  identity holds:
\begin{align}
  \nonumber &\int_Y\left( D(y)\nabla_y w_{i, n-1}^{h} (\tilde{x}, y) - G(u^{H}_{n-1}(\tilde{x})) B(y) w_{i, n-1}^{h} (\tilde{x}, y)\right)  \cdot \nabla_y v(y) \di{y} \\
 \label{microweak_with_lagrange2}  & \hspace{2cm} +  \int_{\Gamma_N} D(y)e_i\cdot n_y v \di{\sigma} + \int_Y d w_{i,n-1}^{h}(y) \di{y} + \int_Y c v \di{y}    =  \int_Y  \nabla_y D(y) \cdot  e_i v \di{y},  
 \end{align}
 for all $ (v, d) \in \mathbb{W}_h \times \mathbb{R}, \;\;i \in \{1, 2\}$.

   For the numerical experiments here, we choose $G(u):= 1-2u$.     
   This choice for $G(\cdot)$ is taken from the upscaled model given in \cite{raveendran2023homogenization}. \SN{For nonlinear choices of the interaction function \(G(\cdot)\), additional numerical simulation results obtained using scheme 2, with and without precomputing, are presented in Appendix \ref{appendix2}. The corresponding results obtained using scheme 1 can be found in \cite{raveendran2023strongly}.} Note that in \cite{raveendran2023homogenization} the authors study the homogenization for the large drift model of the form 
   \begin{equation*}
          \partial_t u^{\varepsilon} +\mathrm{div}(-D^{\varepsilon}\nabla u^{\varepsilon}+ \frac{1}{\es}B^{\varepsilon}u^\es(1-u^{\varepsilon}))=f^{\varepsilon}
   \end{equation*}
   in which the upscaled model has similar structure of Problem $(P)$ with $G(u)=1-2u$.
   The diffusion matrix is  chosen as 
 \begin{align*}
 D(y) :=
 \begin{bmatrix}
2 + \sin(\pi y_1)\sin(\pi y_2) & 0\\
0 & 2+ \sin(\pi y_1)
\end{bmatrix},\quad y=(y_1,y_2)\in Y.
 \end{align*}
For the precomputing step, we solve the weak formulation of the auxiliary cell problem \eqref{auxiliary1}--\eqref{auxiliary3} obtained by replacing $G(\cdot)$ with parameter, $p$, in \eqref{microweak_with_lagrange} and \eqref{microweak_with_lagrange2}. We solve for a total of $201$ different values of $p$ ranging from $-10^{11}$ to $10^{11}$, including $101$ points between $-10$ to $10$. 
The motivation to choose more points near $0$ came from the fact that a dynamic behavior of the dispersion tensor was observed in the vicinity of $0$ in our recent work \cite{raveendran2023strongly}. It is worth noting that for a given velocity field $B$, the parameter $p$  effectively acts as a local Peclet-type number since, depending on its size, it either weakens or strengthens the macroscopic drift effect. Our numerical results in Figure \ref{Fig:Diffusion_tensor} show that the dispersion effect is enhanced near zero for the given velocity field.  This observation closely aligns with the discussion in Section 5 of \cite{ALLAIRE20102292}.

\begin{figure}[ht]
		\centering
\includegraphics[width=0.45\textwidth]{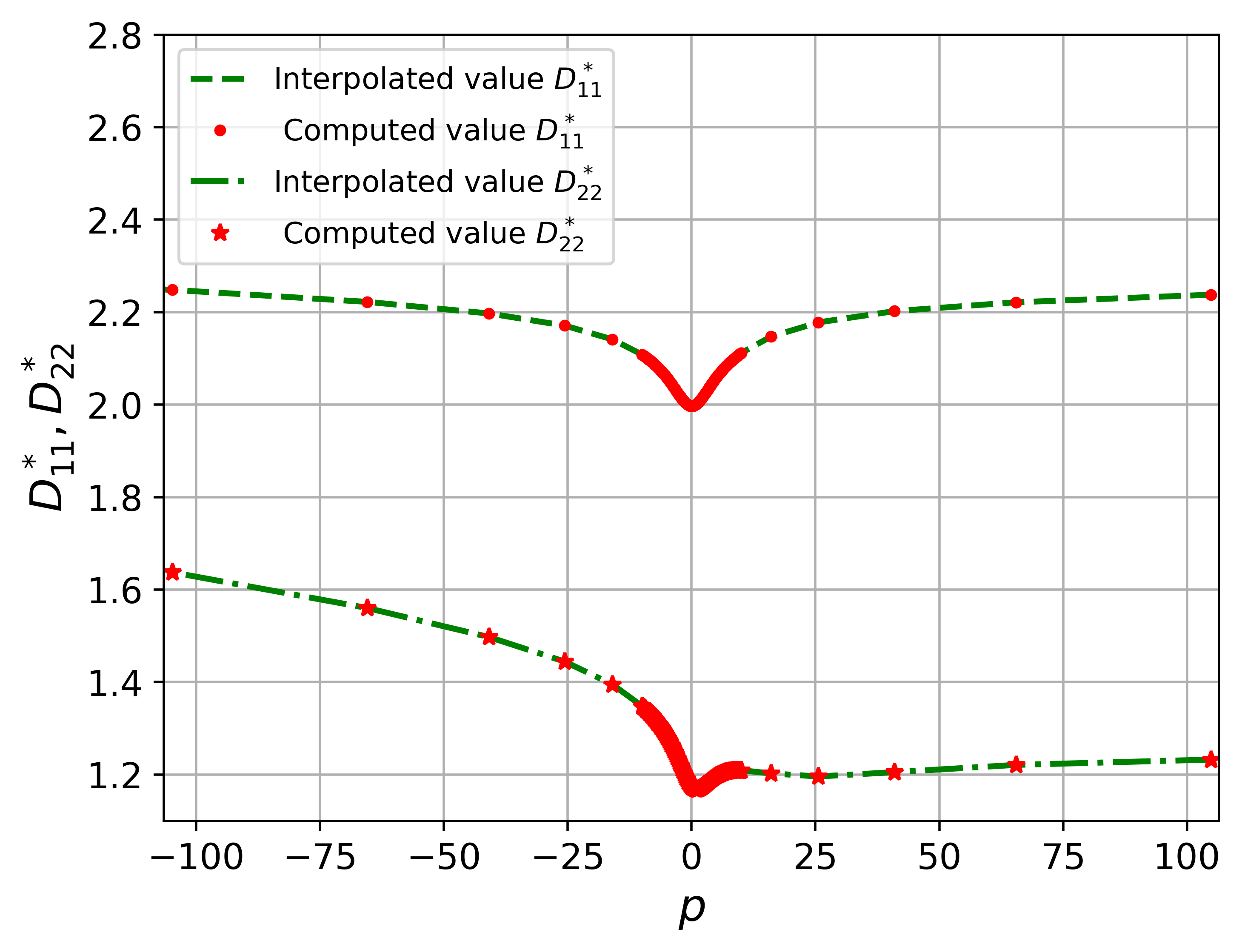}
		\hspace{0.01cm}
\includegraphics[width=0.47\textwidth]{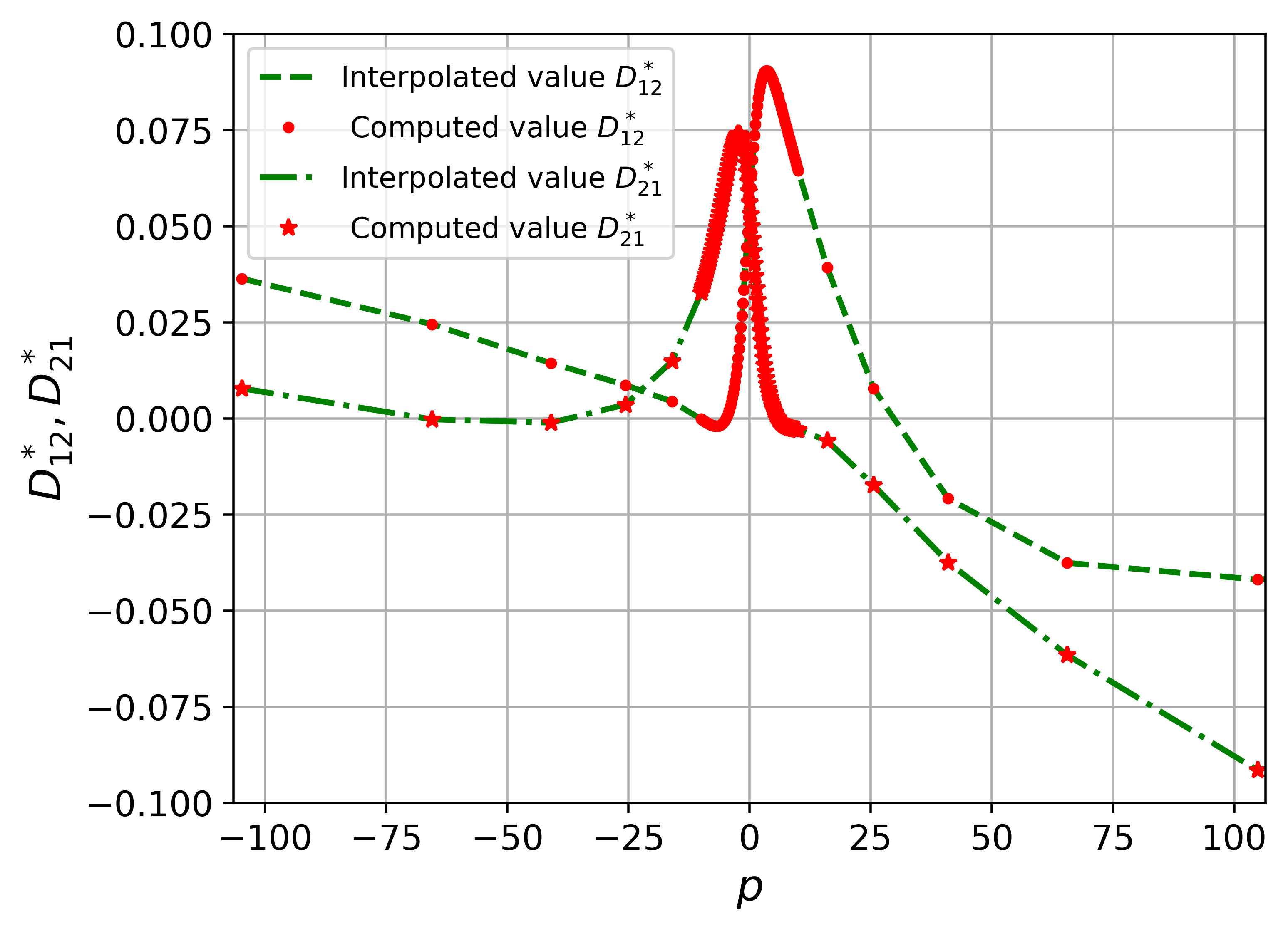}
		\caption{Computed values for the entries of the dispersion tensor 
 $D^*$ for different values of $p$ and its interpolated values: The main-diagonal entries (left) and off-diagonal entries  (right).}
		\label{Fig:Diffusion_tensor}
	\end{figure}
In Figure \ref{Fig:Diffusion_tensor}, we show the result for the computed values of each entry of the dispersion tensor  $D^*$ and its linearly interpolated values, named  $D^{\rm int}$, which we will use later to solve the macroscopic equations. 
  
Once the dispersion tensor $D^*$ or the precomputed dispersion tensor $D^{\rm int}$ is computed, we proceed to solve the macroscopic problem \eqref{homeq1}-\eqref{macro_initial_con}. Although rectangular domains do not meet the boundary regularity assumptions mentioned in Section \ref{assumption}, we set the macroscopic domain as $\Omega = (0, 1)\times (0,2)$ for the convenience of the simulation work. We choose the initial profile to be 
\begin{align*}u_0(x_1, x_2) :=  \begin{cases} \exp{(-10((x_1-0.5)^2 + (x_2-0.5)^2))}, \;\;\text{if} \;\; (x_1, x_2) \in  \mathcal{E}_{0.25, 0.25}((0.5, 0.5)),\\
0, \;\;\text{otherwise},
\end{cases} \end{align*} 
and the source term to be
 \begin{align*} f(x_1, x_2) := \begin{cases} 1000, \;\;\; \text{if} \;\;\; (x_1, x_2)\in \mathcal{E}_{0.25, 0.25}((0.5, 0.5)),\\
0,  \;\;\;\text{otherwise}.
\end{cases}\end{align*}

The initial guess for the iteration scheme, i.e., scheme $1$,  is chosen the same as the initial condition i.e.,  $u^0(t, x) = u_0(x)$. 
We continue the iteration until a maximum number of iterations is achieved or the error $e^k:=||u^{k+1, H}-u^{k, H}||_{L^2(S;  L^2(\Omega))}$ falls below a tolerance, $\epsilon$.
In our simulation, we choose $\epsilon=10^{-7}$. As we are interested in investigating the macroscopic behaviour of the solution, we define the iteration error and the tolerance based on the macroscopic solution. However, it is possible to define these error estimators based on the microscopic solution.

While the precomputing strategy is expected to significantly improve the computational efficiency, it is essential to ensure that the additional error introduced by interpolation does not adversely affect the accuracy of the simulation.
This was investigated analytically in Theorem \ref{Thrm:InterpErrorEstimate}, but we now study this result numerically. 
To proceed with our investigation,  we solve the macroscopic problem with scheme $1$ using both $D^*$ and $D^{\rm int}$. 
We set the maximum number of iterations to $10$. However, the error $e^k$ dips below the tolerance value $\epsilon$ in the $7$th iteration.   
The values of the errors $e^k$ and $e^{k+1}/e^{k}$ are listed in Table \ref{tab:error_iter_scheme}. 

\begin{table}[h]
\centering
\begin{tabular}{c | cc | cc}
\hline
\multirow{2}{*}{k}  & \multicolumn{2}{c}{Scheme $1$}    & \multicolumn{2}{c}{\begin{tabular}[c]{@{}c@{}}Scheme $1$\\ (Precomputing)\end{tabular}} \\ \cline{2-5} 
                   & $e^k$  & $e^{k+1}/e^{k}$ & $e^{k}$  & $e^{k+1}/e^{k}$   \\ \hline
                 0  &  15.716763     &  0.087180      &    15.716707     & 0.086704     \\ 
                 1  &  1.370196     &  0.021651        &   1.362716      &  0.021800    \\ 
                 2  &  0.029666     &   0.015369       &    0.029708      &   0.015469   \\ 
                 3  &  0.000455     &   0.022186       &     0.000459    &   0.021562   \\ 
                 4  & 1.011 e${-05}$ &   0.033969      &  9.909 e${-06}$  & 0.033860    \\ 
                 5  & 3.436 e${-07}$  &   0.031789       &  3.355 e${-07}$  &  0.031190    \\ 
                 6  &  1.092 e${-08}$     &          &     1.047 e${-08}$    &      \\ \hline
\end{tabular}
\caption{Computation values of $e^k$  with $D^*$ and precomputed $D^{\rm int}$. We choose $4096$ DOFs for the macroscopic domain. \SN{We set the solver tolerance $\epsilon$ equal to $10^{-7}$ throughout the simulations.}}
    \label{tab:error_iter_scheme}
\end{table}

 A noticeable trend is observed in Table \ref{tab:error_iter_scheme}  where the error $e^k$ monotonically decreases as the number of iterations increases and the ratio  $e^{k+1}/e^{k}$ is less than $1$. As we are only reporting the sequential errors in the macroscopic solution at the different iterates, the errors in Table \ref{tab:error_iter_scheme} are almost identical and only differ slightly due to the additional error introduced from precomputing.
 
The precomputed dispersion tensor $D^{\rm int}$ is now utilized in scheme $2$ to solve the macroscopic problem. 
 We begin with comparing the temporal evolution of $D^*(t, \tilde{x})$ and $D^{\rm int}(t, \tilde{x})$  at a fixed, arbitrarily chosen, spatial point  $\tilde{x} = (0.5714,  0.5714) \in \Omega$. 
The outcomes for all four entries of $D^*$ and $D^{\rm int}$ are shown in Figure \ref{Fig:Comp_Diffusion_tensor}. 
\begin{figure}[ht]
		\centering
\includegraphics[width=0.44\textwidth]{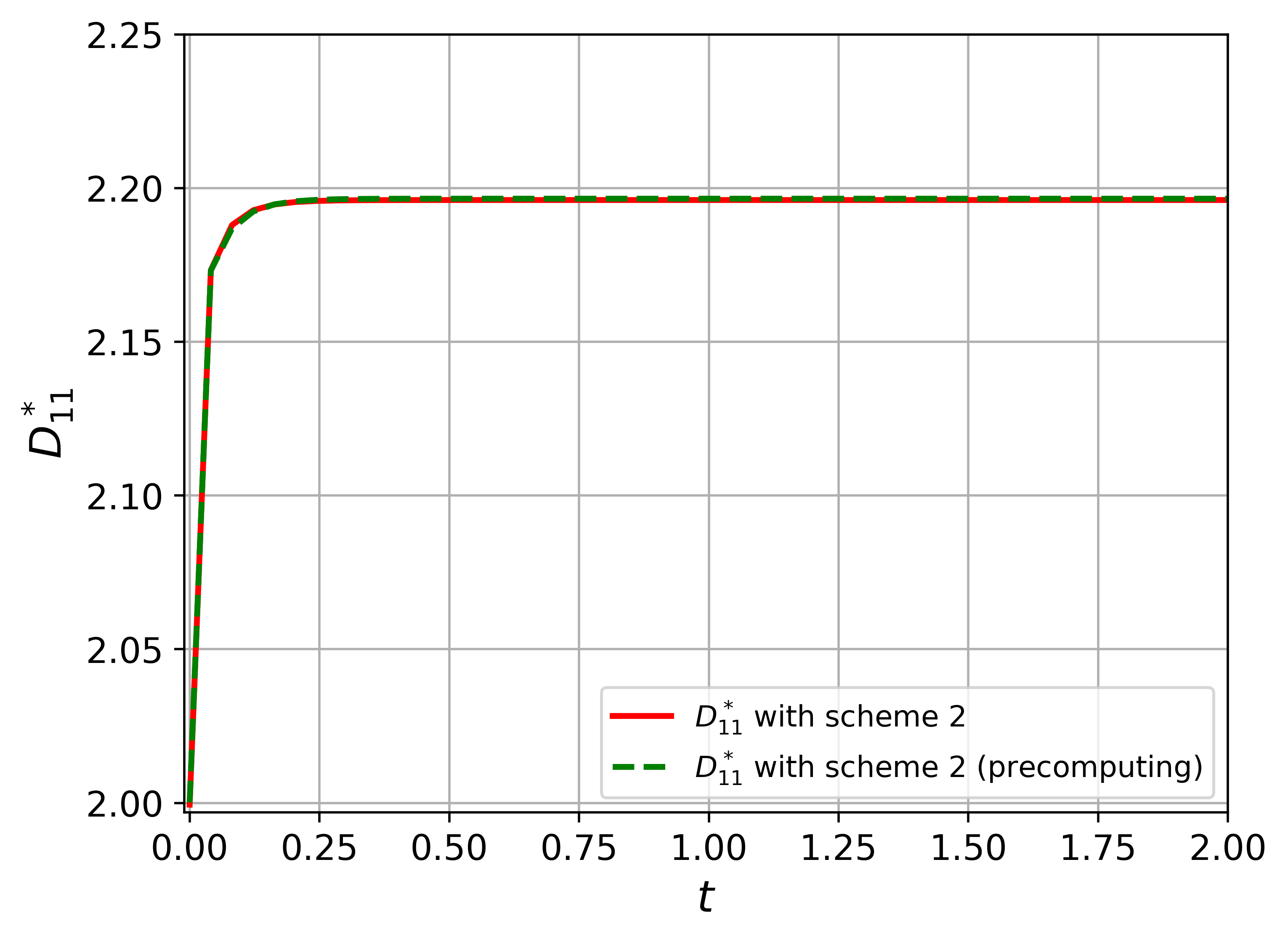}
		\hspace{0.01cm}
\includegraphics[width=0.44\textwidth]{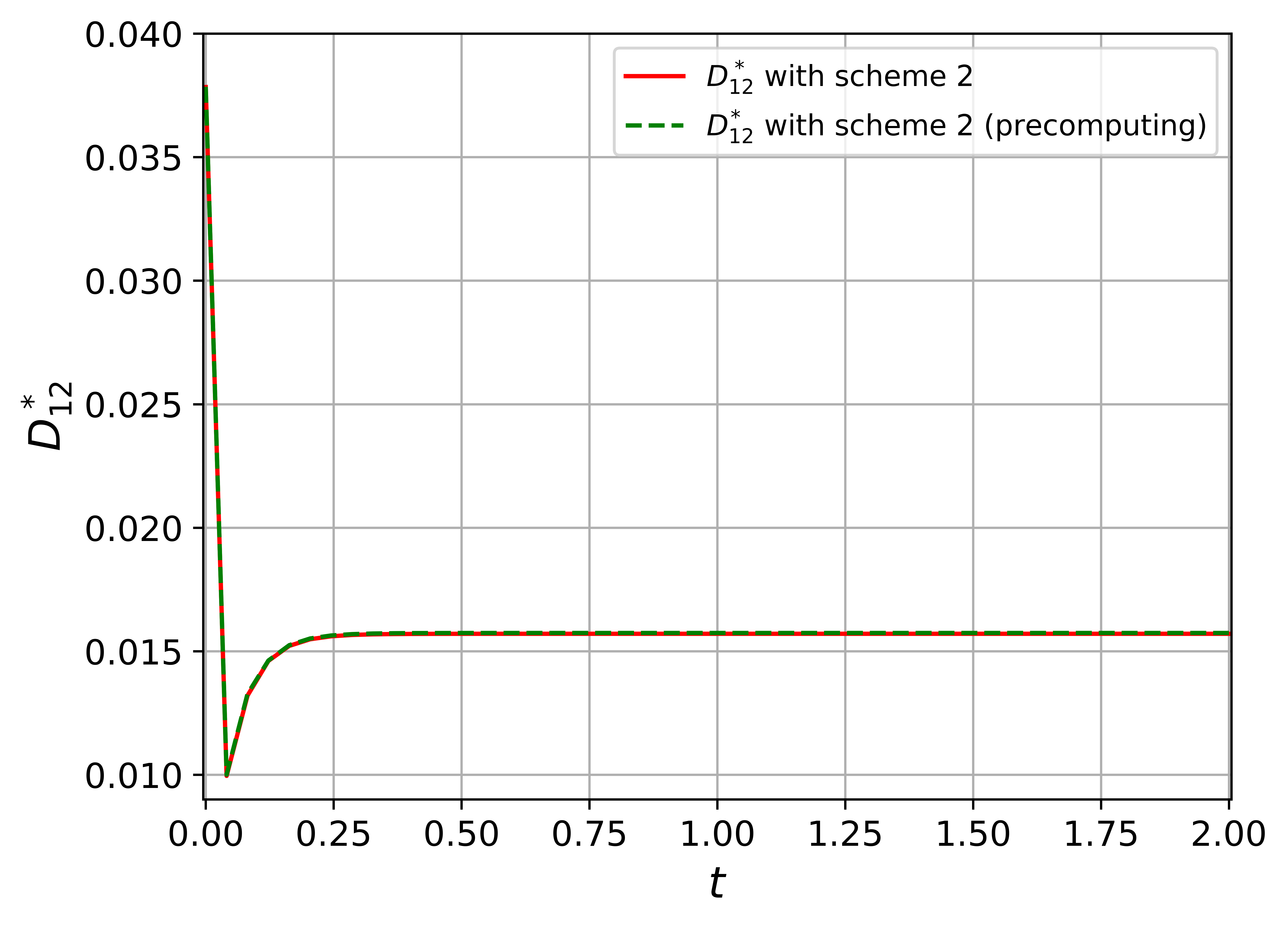}\\
\includegraphics[width=0.44\textwidth]{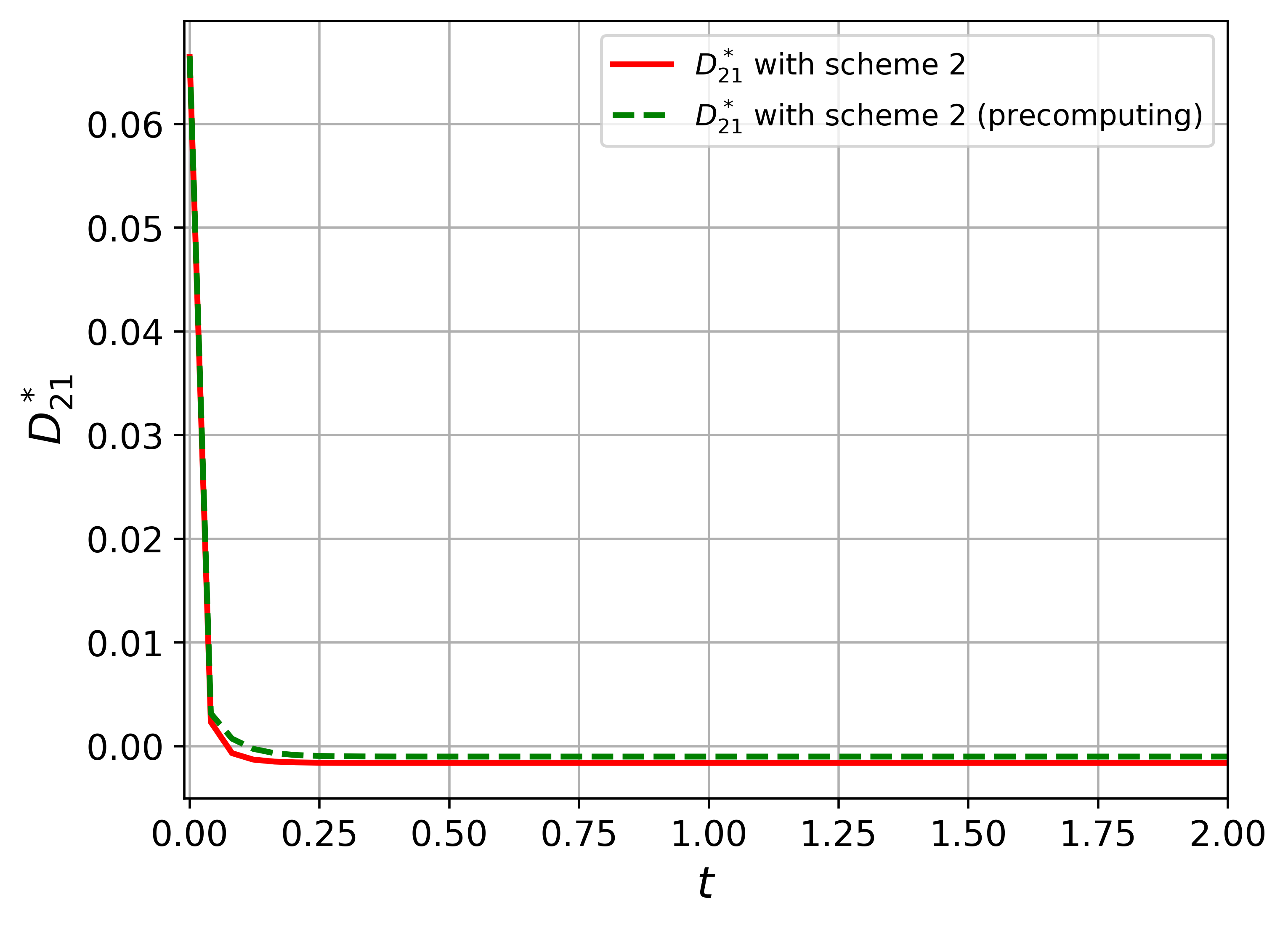}
		\hspace{0.01cm}
\includegraphics[width=0.44\textwidth]{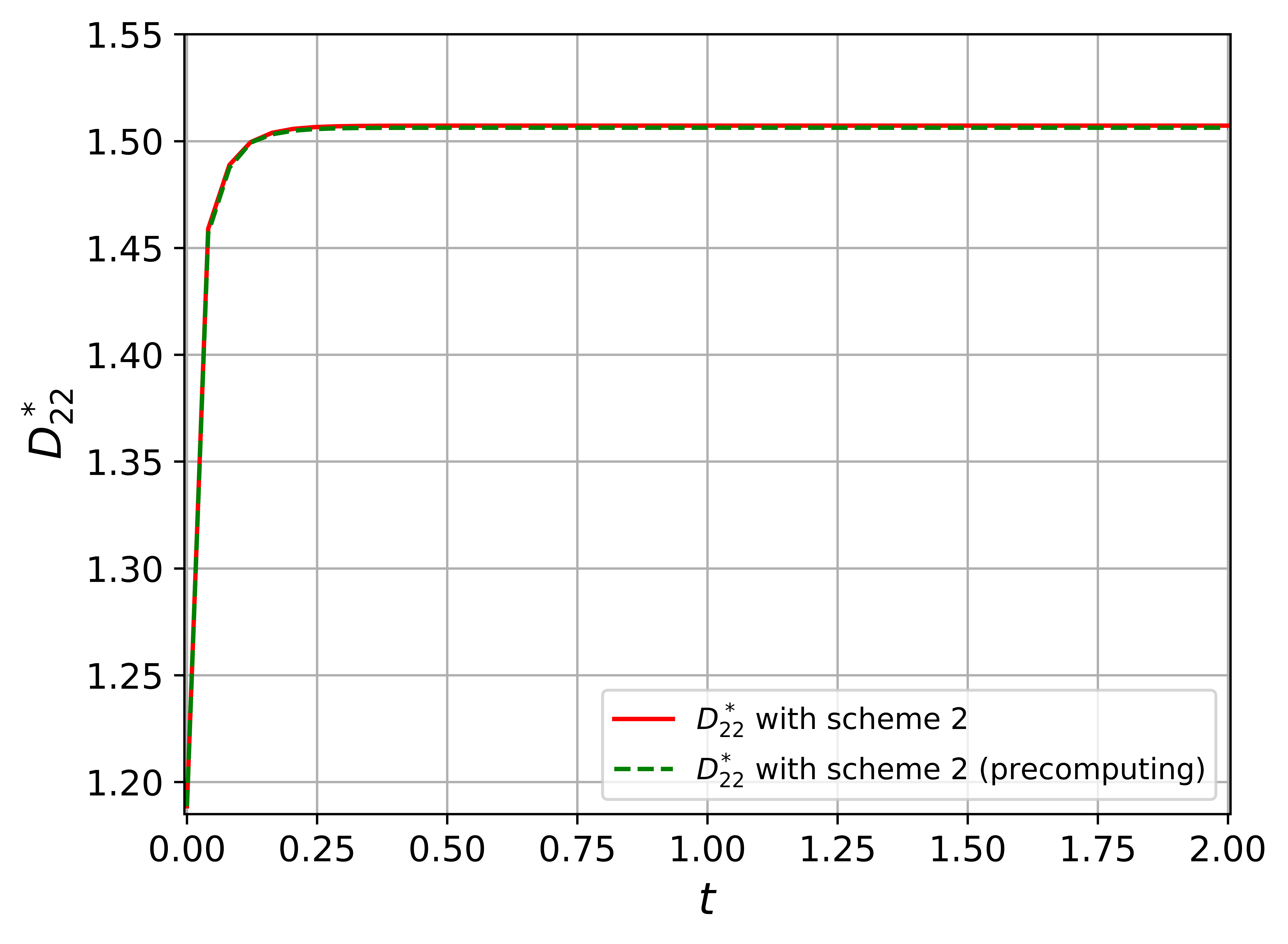}
\caption{Comparison of the time evolution of the dispersion tensor 
 $D^*$  and $D^{\rm int}$ at the point $(0.5714,  0.5714)$.}
\label{Fig:Comp_Diffusion_tensor}
	\end{figure}
We can see in Figure \ref{Fig:Comp_Diffusion_tensor} that all entries of $D^{\rm int}$ are in good agreement with the entries of $D^*$, suggesting that the accumulated error due to the interpolation is manageable.
  
To compare the macroscopic solution with and without precomputing,  we point out the plots for the macroscopic solution of scheme $2$ with  $D^*$ and $D^{\rm int}$ (with precomputing) at $T=2$ in Figure \ref{macrosolution_scheme2} where we also plot the pointwise difference between the two solutions, see the right-most plot. This simulation is done using 2016 DOF in the macroscopic domain. 

\begin{figure}[ht]
		\centering	
    \includegraphics[width=0.35\textwidth]{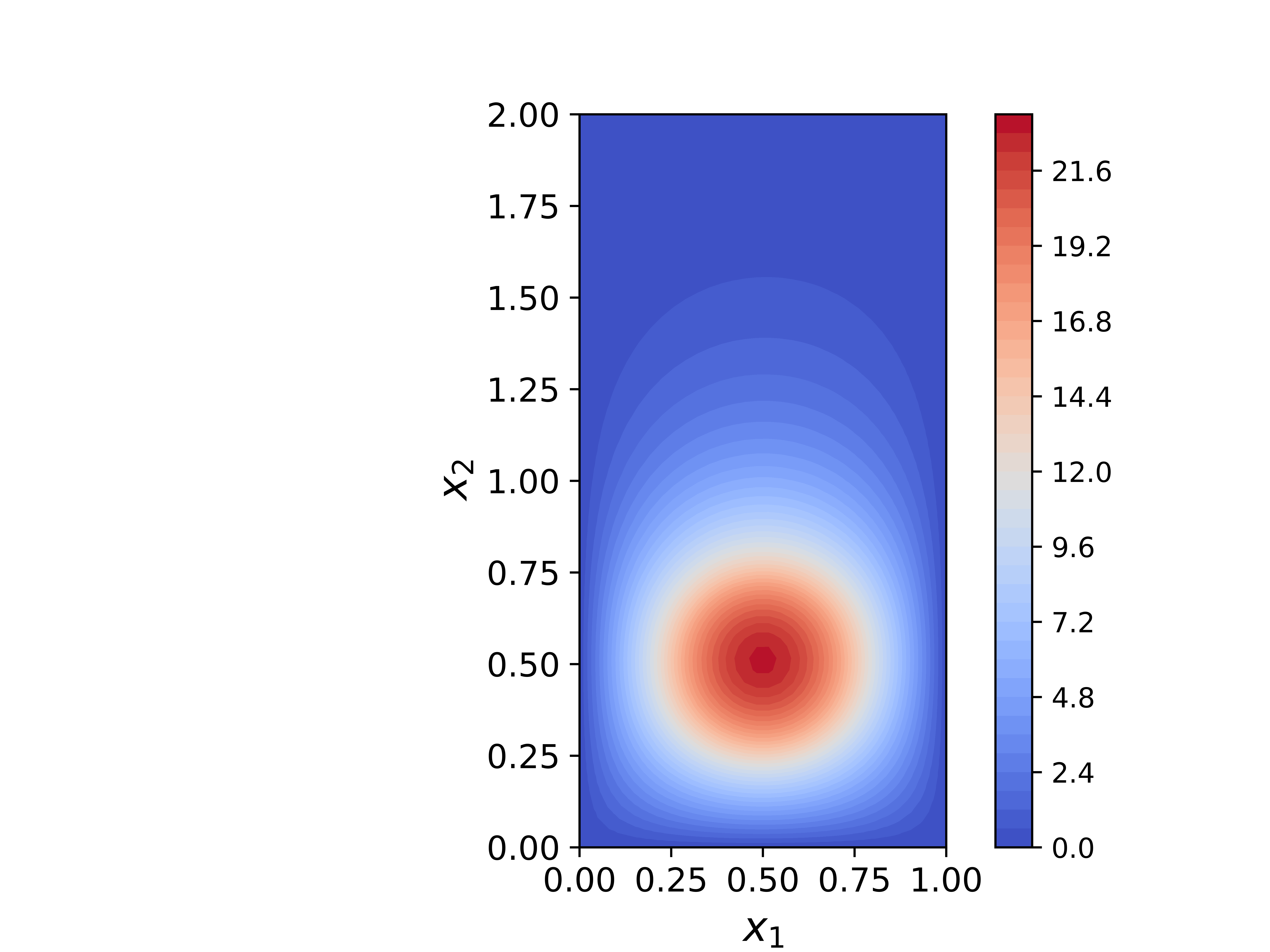}
		\hspace{-0.9cm}
       \includegraphics[width=0.35\textwidth]{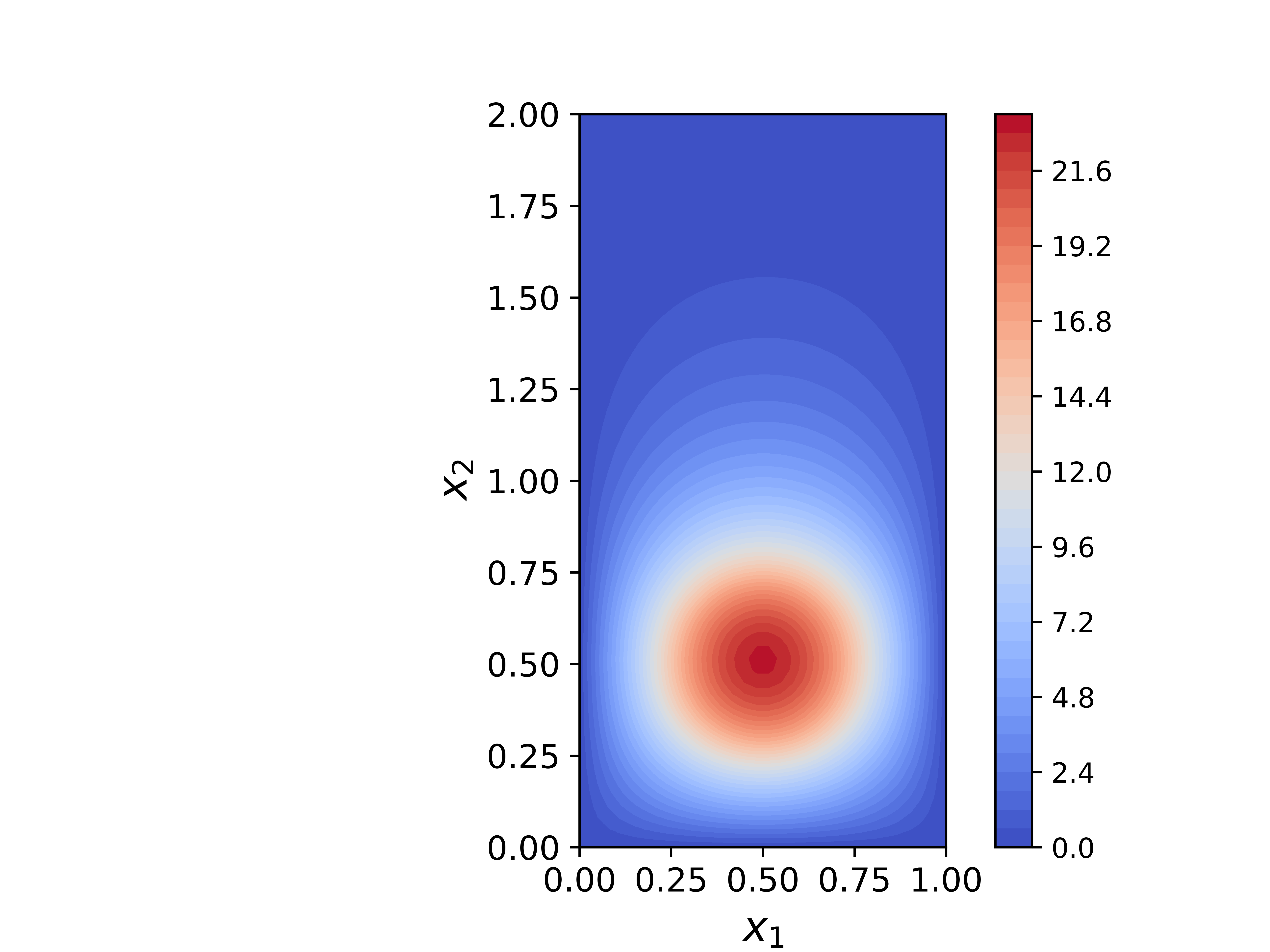}
        \hspace{-0.9cm}
        \includegraphics[width=0.35\textwidth]{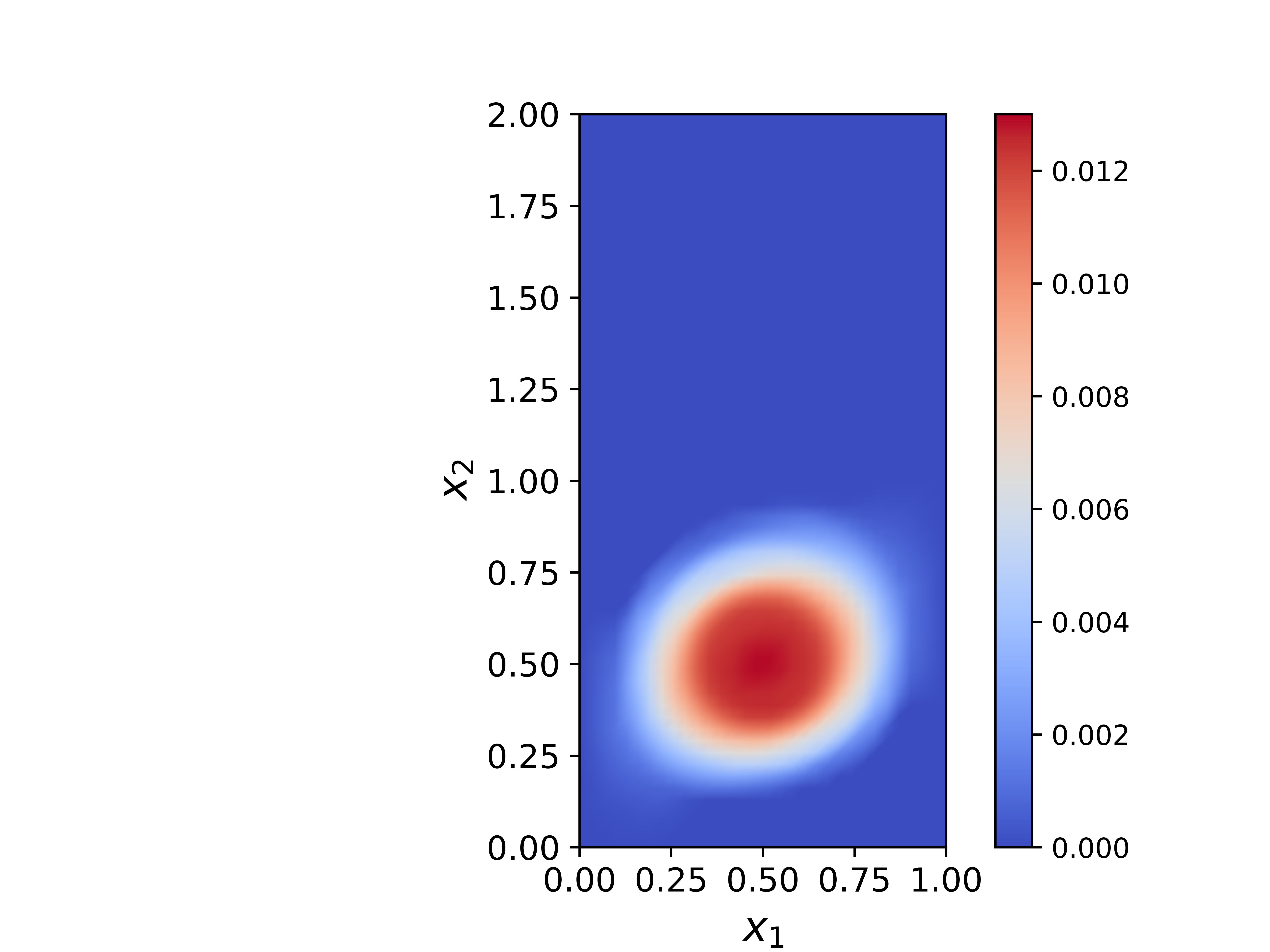}
		\caption{Concentration profile approximated via scheme $2$ (left) and via scheme $2$ with precomputing (middle). The pointwise difference between these two approximations at $T = 2$ and $M = 50$ (right).}
		\label{macrosolution_scheme2}
\end{figure}

\subsection{Computing time and error analysis} \label{erroranalysis}
In this section,  we compare the computing time and the experimental convergence errors of both schemes, with and without precomputing. To begin with,
we calculate the errors and computing time of both schemes for different choices of macroscopic mesh sizes and compare them in Table \ref{tab:errorandtime}. Comparing these two numerical schemes, the computing time for scheme $2$ is less than that of scheme $1$. 
This is to be expected as scheme $2$ does not have the additional overhead from the Picard iteration steps.  
Also, perhaps as expected, precomputing helps to save computing time for both schemes. 
The precomputing process takes approximately $46$ seconds to solve the $201$ auxiliary problem with $813$ DOFs in the microscopic domain. 
Excluding the precomputation time costs, we see in Table \ref{tab:errorandtime} that for each choice of the macroscopic meshes,  the total cost to compute the macroscopic solution, using the precomputed dispersion tensor, is at least $100$ times smaller than without precomputing. 

\begin{table}[ht]
\centering
\begin{tabular}{lllll}
\hline
\multirow{2}{*}{Macro DOFs}                     & \multicolumn{2}{c}{Scheme $1$}                                            & \multicolumn{2}{c}{Scheme $1$ (precomputing)}         \\ \cline{2-5} 
                                         & Errors                     & Computing time (s)                     & Errors                     & Computing time (s) \\ \hline
16                                       &     4.8023658                        &   396.91                                         &          4.804463                  &     2.25                   \\
64                                       &    1.6308094                         &   1781.25                                         &            1.632296                &    4.42                    \\
256                                      &   0.4155008                         &   7059.71                                         &             0.416212              &     11.70                   \\
1024                                     &   0.1678484                         &  28417.18                                          &               0.1671295             &    53.35                    \\
4096                       &                            &   113488.39                                         &                            &                  189.84      \\ \cline{2-5} 
\multicolumn{1}{c}{\multirow{2}{*}{Macro DOFs}} & \multicolumn{2}{c}{Scheme $2$}                                            & \multicolumn{2}{c}{Scheme $2$ (precomputing)}         \\ \cline{2-5} 
\multicolumn{1}{c}{}                     & \multicolumn{1}{c}{Errors} & \multicolumn{1}{c}{Computing time (s)} & \multicolumn{1}{c}{Errors} & Computing time (s) \\ \cline{2-5} 
16                                       &    4.8023659                        &      70.52                                      &                        4.804463    &       0.36                 \\
64                                       &   1.6308094                         &      278.65                                      &                        1.632296290    &         0.48               \\
256                                      &   0.4141075                        &    1112.10                                        &                       0.416212306     &       1.30                 \\
1024                                     &  0.1667797                          &     4345.78                                       &                       0.16712954     &       4.64                 \\
4096                                     &                            &         17013.41                  &                        &   20.92                    
\end{tabular}
\caption{\label{tab:errorandtime} Errors and  computing time of the schemes for $T=2$ with $M= 20$. The microscopic problem is solved with $816$ DOFs for both schemes.  The precomputed $D^{\rm int}$ is obtained by solving $201$ auxiliary problems and it takes approximately $46$s, and hence, for a fair estimate of total time, one should add precomputing time of 46s to both schemes with precomputing.}
\end{table}
The analytical solution to the problem is unavailable and, therefore, we compare two macroscopic solutions, one solved on a more refined mesh than the other, to test the accuracy and convergence of our numerical schemes.  
To compute the error, we first fix the microscopic mesh with  $813$ DOF to solve the Stokes problem and microscopic problem.  We then fix the time mesh by taking $M=20$ and solve the macroscopic problem on two different meshes, one with the four times DOFs as the other one, and then compare the results on $L^2(S; L^2(\Omega))$ to compute the resulting error. Indeed,  a refinement step (corresponding roughly to a macroscopic mesh size divided by two), starting with the initial macroscopic mesh with
$16$ DOF, leads to $64$ DOF and $256$ respectively for the next two refinements as mention in Table \ref{tab:errorandtime}. 
 We observe that the $L^2$ errors have the expected decay rate as seen in Table \ref{tab:errorandtime}. We show in Figure \ref{loglogplot_errors} the order of convergence \SN{(OC)} concerning macroscopic space mesh size.  
Because of the similar behavior of the error values with scheme $1$ and scheme $1$ (precomputing), we skip to display the error plots for them.

\begin{figure}[hbt!]
		\centering	
    \includegraphics[width=0.45\textwidth]{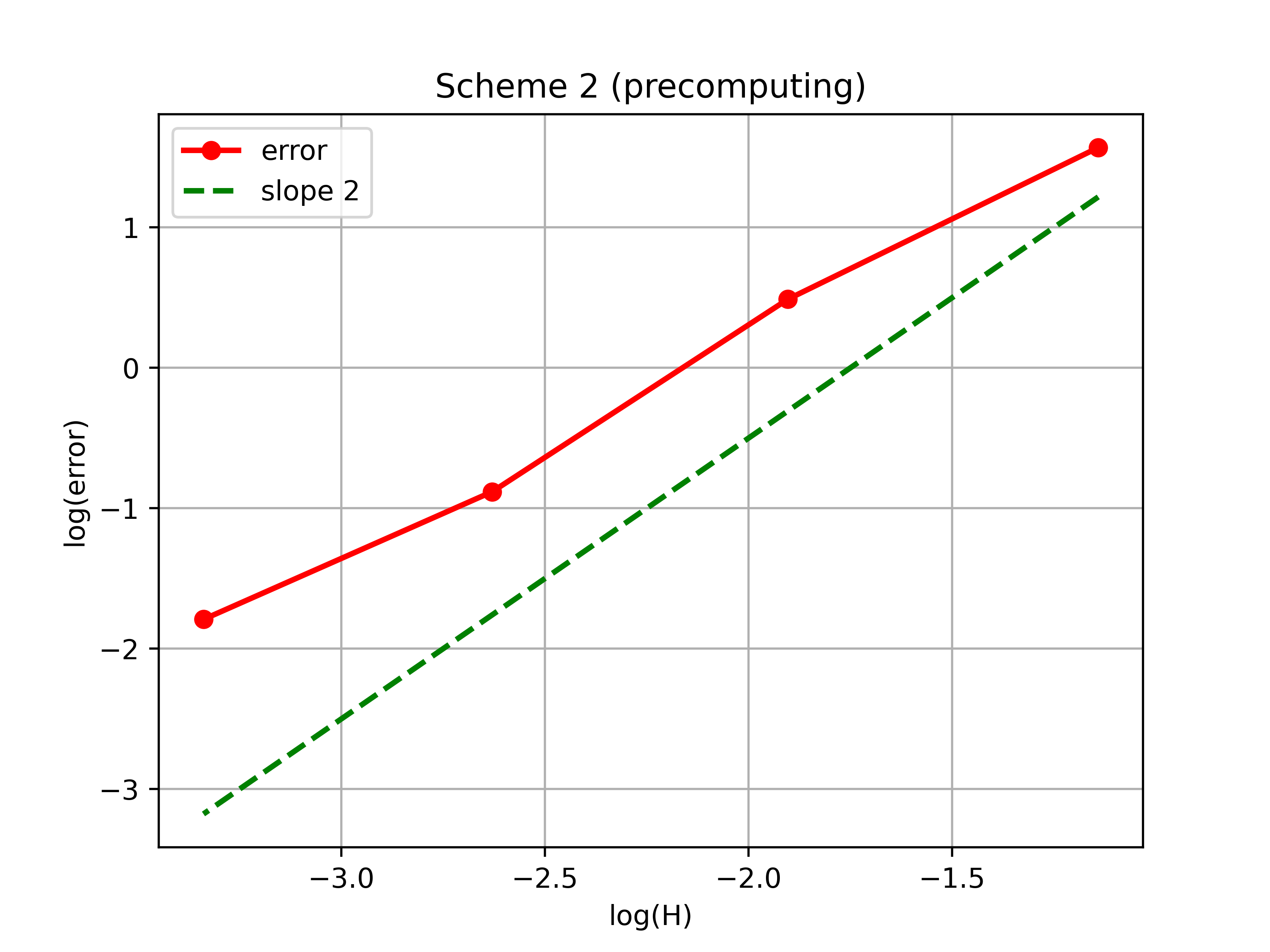}
    \hspace{0.001cm}
    \includegraphics[width=0.45\textwidth]{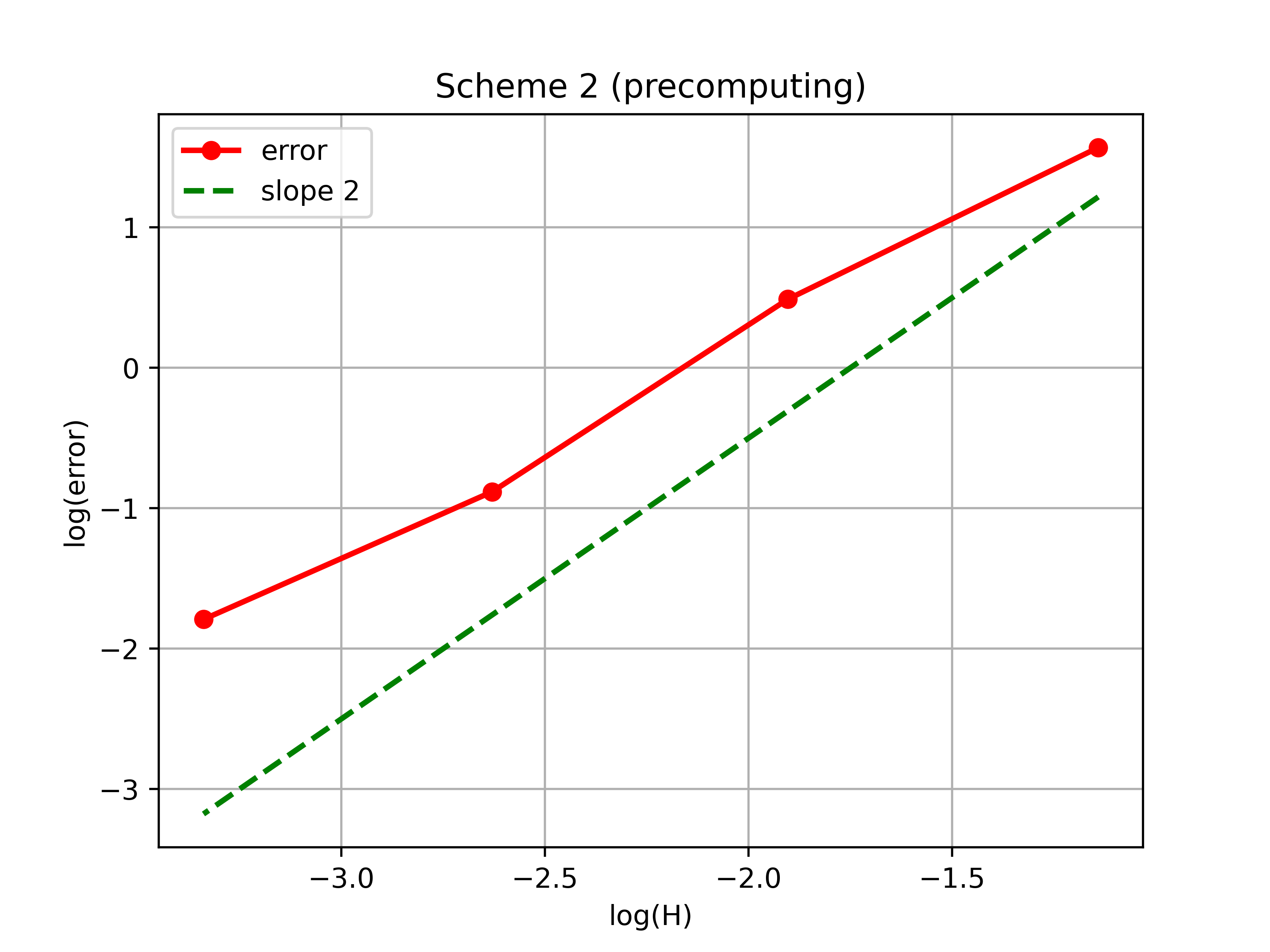}
		\caption{Evolution of errors: log-log plot of $L^2$ error as a function of $H$ with scheme $2$ (left)  and scheme $2$ precomputing (right). The final time is chosen $T = 2$ with $M=20$. The microscopic mesh is fixed with   $813$ DOF  and only the macroscopic mesh is refined by taking the DOF  mentioned in Table \ref{tab:errorandtime}. The dashed line is a reference line with a slope equal to $2$. }
		\label{loglogplot_errors}
\end{figure}
Our study shows a similar error behavior for both schemes with and without precomputing. However, as expected, scheme $2$ with precomputing is the computationally cheapest one. 
We now present our numerical exploration of the temporal order of convergence for scheme $2$ with precomputing.  
To estimate the \SN{OC} in time, we compute the error behavior for different time steps with a fixed macroscopic and microscopic discretization. 
We keep a microscopic mesh fixed with $813$ DOF to solve the Stokes problem and the auxiliary problem as mentioned earlier.  We then fix the macroscopic mesh with $2500$ DOF and compute the macroscopic solution on different time meshes with $M= 320, 640, 1280, 2560, 5120$. 
We then compare the two consecutive macroscopic solutions on  $L^2(S; L^2(\Omega))$ norm to calculate the resulting error.  The comparison between the log-log plot for the error {\it versus} time mesh size and a reference line with slope $1$ is depicted in Figure \ref{loglogplot_errors_refine_time}. As expected from treating the time discretization of a linear parabolic equation by an implicit  Euler scheme (see for instance \cite{thomee2007galerkin}), the order of convergence in time is $1$.

\begin{figure}[hbt!]
		\centering	
    \includegraphics[width=0.45\textwidth]{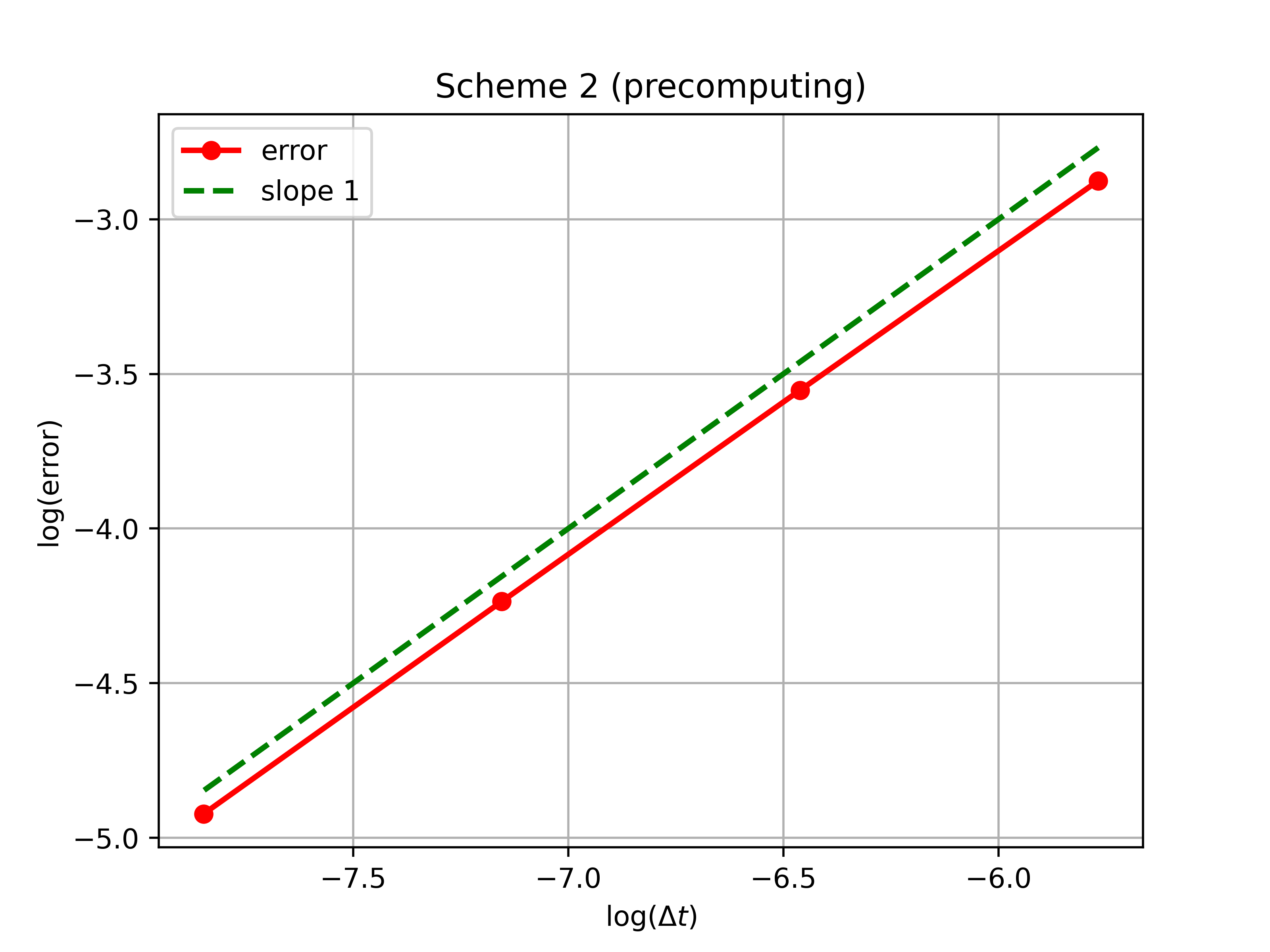}
		\caption{Evolution of errors: log-log plot of $L^2$ error as a function of $\Delta t$ with  scheme $2$ (precomputing).  The final time is chosen as $T = 2$. The mesh refinement is done only in time intervals by choosing $M= 320, 640, 1280, 2560, 5120$.  The dashed line is a reference line with a slope equal to $1$. }
		\label{loglogplot_errors_refine_time}
\end{figure}
We now explore the build-up of errors for scheme $2$ (precomputing) when we refine the microscopic, macroscopic and time mesh at the same time. 
At this time, to compute the error, we first solve the microscopic and macroscopic problem with a fixed number of microscopic, macroscopic DOF as well as with a fixed value of $M$ and then we again solve the same microscopic and macroscopic problem on a refined microscopic, macroscopic and time meshes. The refinement is done by increasing DOFs in the macroscopic and microscopic meshes as well as taking two times larger values for  $M$ in the time domain. The chosen values for the DOFs and $M$ are listed in Table \ref{tab:errorandtime_macro_micro_time_refine}.  We then compute the errors on $L^2(S; L^2(\Omega))$ norm by comparing the obtained two macroscopic solutions. 
The corresponding errors have the
expected decay rate as it can be seen in Table \ref{tab:errorandtime_macro_micro_time_refine}. We point out in Figure \ref{loglogplot_errors_refine_all} the log-log scale plot for the error {\em versus} the macroscopic mesh size $H$. Here we compare the results against a reference line with a slope equal to $2$.

\begin{table}[h]
\centering
\begin{tabular}{lllllll}
\hline
Macro DOFs  & Micro DOFs   &   $H$    & $h$  &   $M$    &  Errors   &Computing time (s)     \\ \hline
64          & 692  & 0.31943    & 0.15646 & 320        &    2.245391 &   17.60         \\
256         &  836  & 0.14907    & 0.07918 &640     &    0.885786 &     71.24            \\
1024        &  1641  & 0.07213   & 0.03977  &1280      &   0.257523&      371.30                \\
4096        & 5288  & 0.03549 & 0.01988 &2560    &   0.072165  &    2789.38                \\
16384       & 20517  & 0.01760   & 0.00994  &5120    &         &                 20622.87   \\
\end{tabular}
\caption{\label{tab:errorandtime_macro_micro_time_refine} Errors and computing time of the schemes 2 (precomputing) for $T=2$. Precomputed $D^{\rm int}$ is obtained by solving the $201$ auxiliary problem.}
\end{table}

\begin{figure}[hbt!]
		\centering	
    \includegraphics[width=0.45\textwidth]{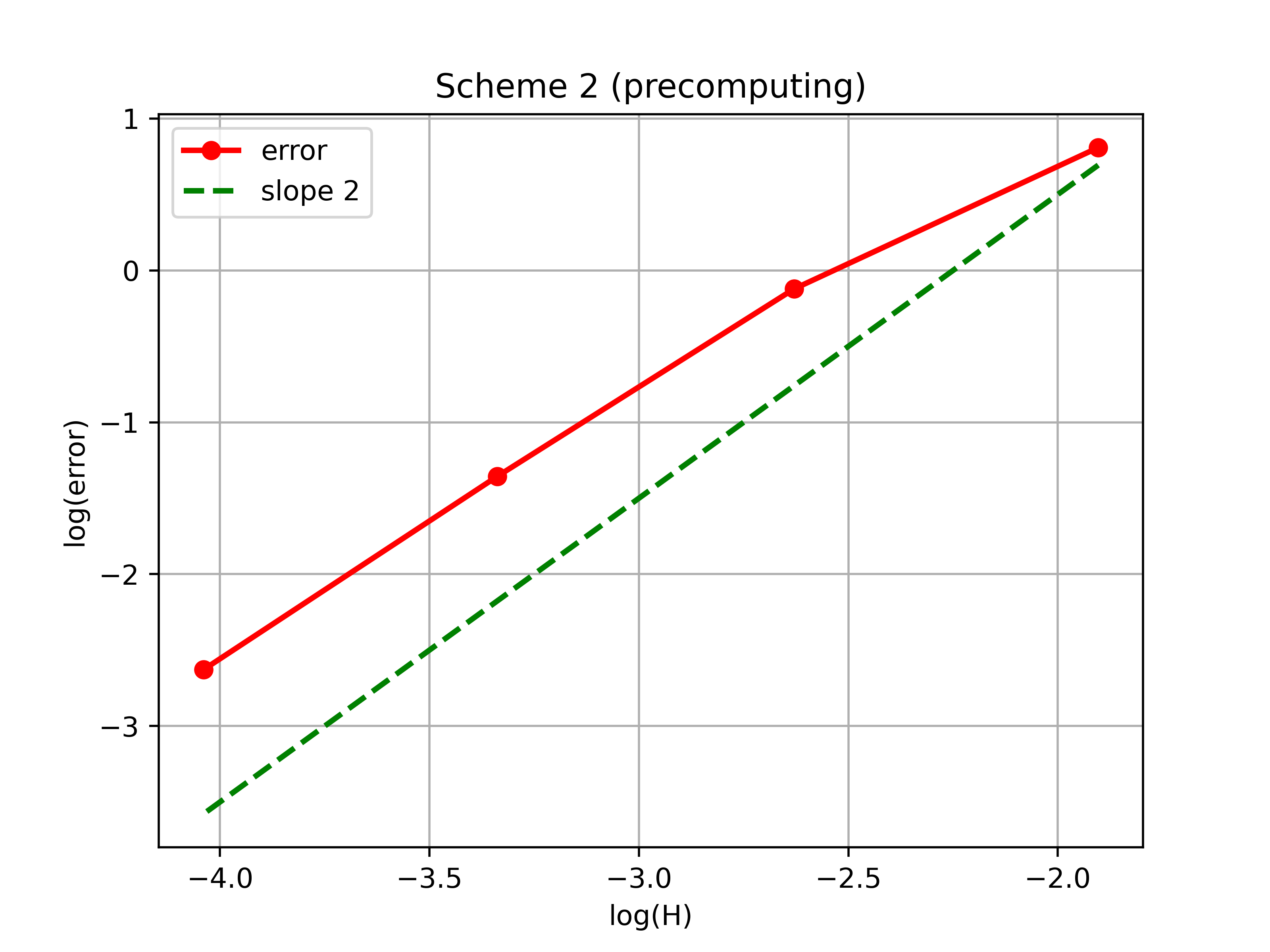}
    \hspace{0.01cm}
    \includegraphics[width=0.45\textwidth]
    {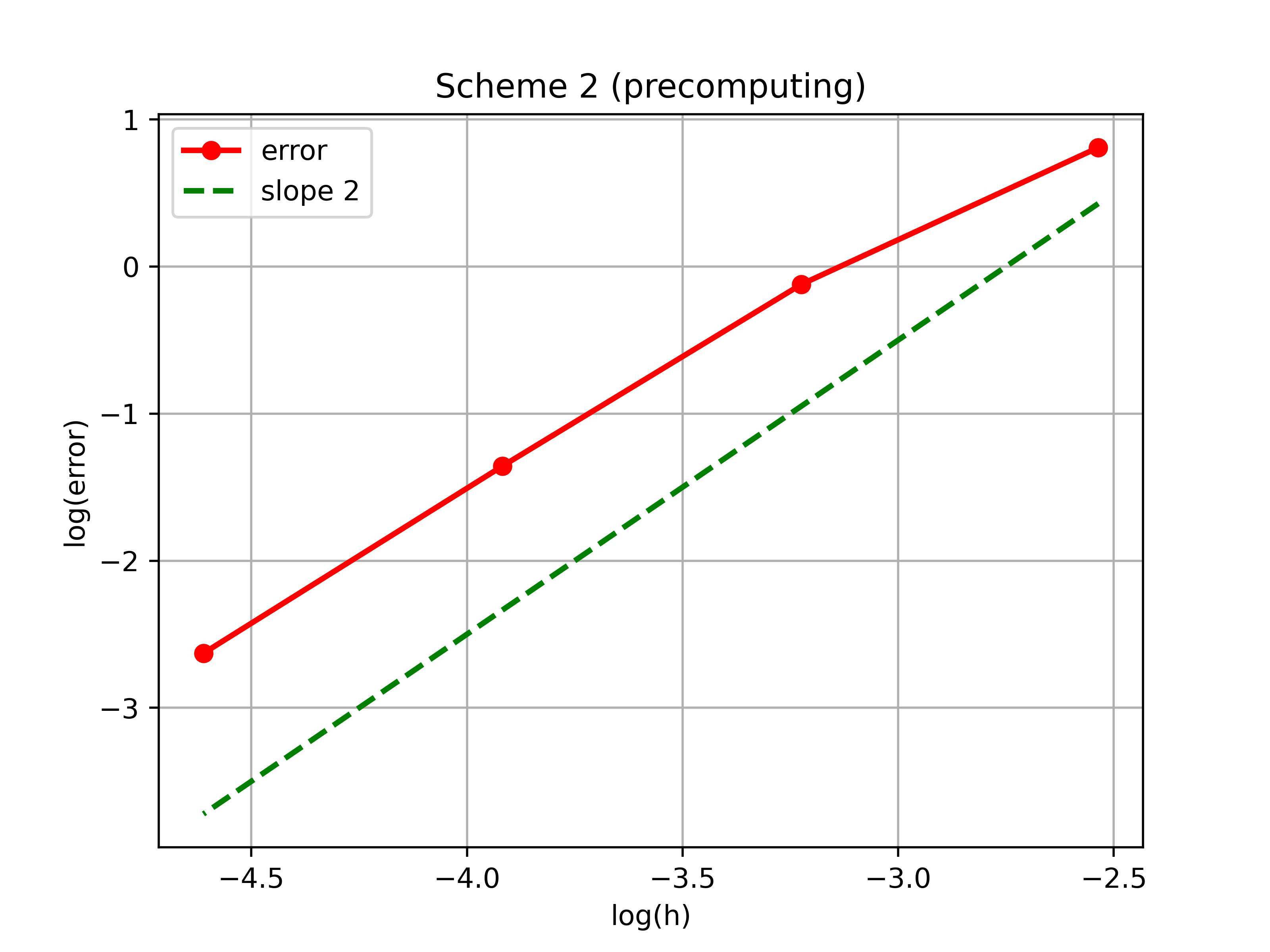}
		\caption{Evolution of errors: log-log plot of $L^2$ error as a function of $H$ with  scheme $2$ precomputing (left). log-log plot of $L^2$ error as a function of $h$ with  scheme $2$ precomputing (right). The final time is chosen as $T = 2$. The mesh refinement is done in time and for both involved space domains as mentioned in Table \ref{tab:errorandtime_macro_micro_time_refine}. The dashed line is a reference line with a slope equal to $2$. }
		\label{loglogplot_errors_refine_all}
\end{figure}

\SN{To assess the influence of the interpolation parameter $\delta$ numerically, we compare the macroscopic solution obtained by scheme 2 with and without precomputing for different choices of $\delta$.
The macroscopic mesh, microscopic mesh, and time step are kept fixed, and as before, the error is measured in the $L^2(S; L^2(\Omega))$ norm. The results are reported in Table~\ref{tab:precompute_delta} which shows the interpolation error decreases monotonically as $\delta$ is refined, with an observed OC close to two for this problem, in agreement with \cref{rem:InterpOC}. 
As expected, the offline precomputing time increases as $\delta$ decreases. Using the precomputed effective diffusion tensor, scheme 2 computes the solution in about $1.5$\,s, whereas scheme 2 without precomputing requires $4045.77$\,s. 
\begin{table}[h]
\centering
\begin{tabular}{ccccc}
\hline
$\delta$ &
Error &
OC &
Offline computing time (s) \\
\hline
$0.1$  & $4.0723\times10^{-6}$ & --    & $40.50$  \\
$0.05$  & $1.0196\times10^{-6}$ & $1.998$ & $79.98$ \\
$0.025$  & $2.4925\times10^{-7}$ & $2.032$ & $159.49$  \\
$0.0125$ & $6.2283\times10^{-8}$ & $2.001$ & $315.93$ \\
\hline
\end{tabular}
\caption{\SN{Errors, order of convergence (OC) and computing time of scheme  2 with precomputing for different values of $\delta$. The simulations are performed with $L= 50$, macro DOFs $=4096$, micro DOFs $=975$, $M= 25$ and $T=1$. The computing time for solving the problem without precomputing is $4045.77$ s.}}
\label{tab:precompute_delta}
\end{table}
}

\subsection{Capturing the penetration depth of diffusing particles into structured materials }\label{application}

In this section, we numerically explore a specific application of Problem $(P)$ for a scenario where we replace the homogeneous Dirichlet macroscopic boundary condition \eqref{macro_boundary_con} on one part of the boundary with a non-homogeneous Dirichlet boundary condition.
This modification allows us to test the applicability of our model and simulation techniques  for capturing the dynamics of penetration of populations of particles in structured materials; compare to the scenario discussed in \cite{nepal2021moving} where many small particles want to ingress into a specific type of rubber-based  material. \V{Note that, even though the original problem is formulated with a homogeneous Dirichlet boundary condition, the nonhomogeneous Dirichlet data presented here can be treated by a standard lifting argument (see \cite[Section 2]{raveendran2022upscaling}  for more detail).}
 
To begin our investigation, we define two specific microscopic geometries.
We denote the microscopic domain $Y$ defined in \eqref{micro_domain1} as geometry $1$.
 For geometry $2$, which we already looked at in \cite{raveendran2023strongly}, we let $Y = (0,1)^2 \setminus  (\mathcal{R}_{1} \cup \mathcal{R}_{2})$ with rectangles $\mathcal{R}_{1}:= [0.1, 0.9] \times [0.1, 0.2]$ and $\mathcal{R}_{2}:= [0.1, 0.9] \times [0.8, 0.9]$. 
To visualize geometries, we refer the reader to Figure \ref{microsolution_scheme2_geometry123}.
We solve the Stokes problem in the respective geometries with the same choice of viscosity and force as defined in \cref{microsolution}. The simulation results for $w_1$ after solving the auxiliary problem \eqref{auxiliary1}--\eqref{auxiliary3} with $p=1$ are presented in Figure \ref{microsolution_scheme2_geometry123}. 

 \begin{figure}[h]
		\centering	
    \includegraphics[width=0.35\textwidth]{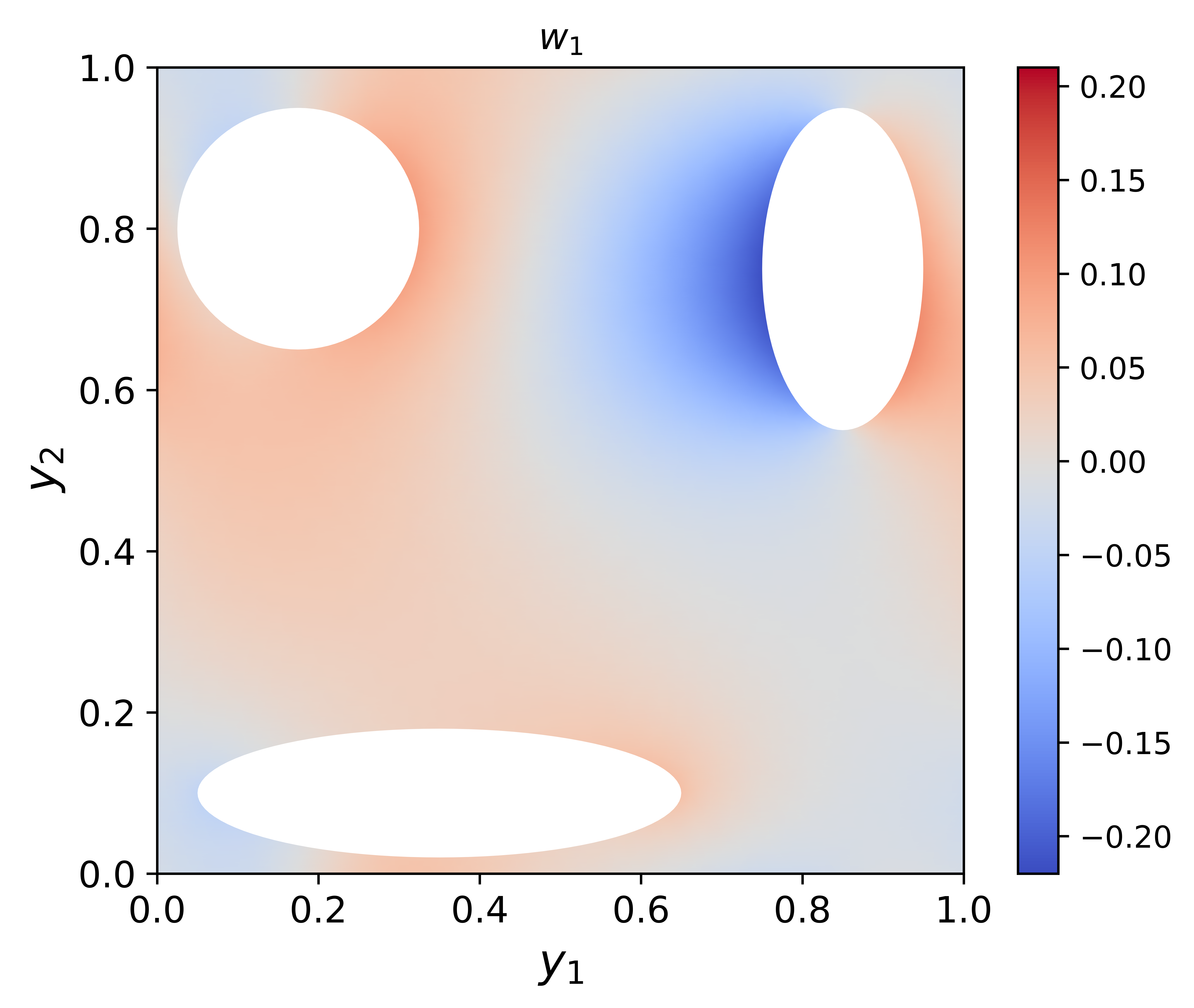}
    \hspace{0.2cm}
    \includegraphics[width=0.35\textwidth]{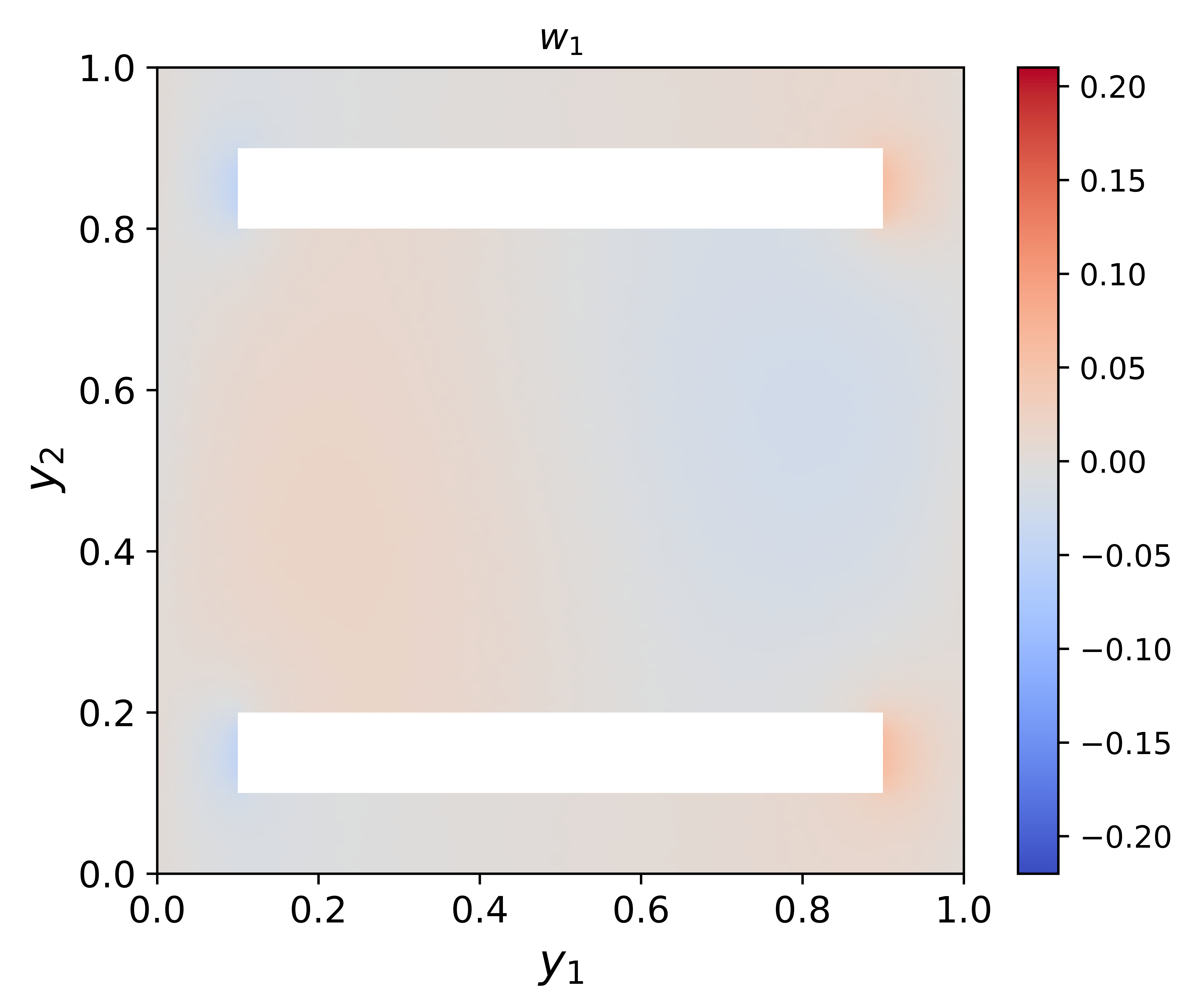}
		\caption{Microscopic solution $w_1$ with geometry $1$ (left) and geometry $2$ (right).}
		\label{microsolution_scheme2_geometry123}
\end{figure}

We choose the unit square domain as a macroscopic domain. Initially, there is no mass of concentration inside the domain, i.e. we set $u_0 = 0$.  We allow some amount of concentration to be present at the bottom part of the boundary, by prescribing the time-dependent non-homogeneous Dirichlet boundary condition $10t/(1+t)$, while the other three parts of the Dirichlet boundary are set to zero. 
 Additionally, the source term $f$  is taken to be zero. Since (for our setting)  
 scheme $2$ with precomputing is computationally the most efficient one, we use it to compute the macroscopic solution; we show the result in Figure \ref{macrosolution_scheme2_non_hom_diri}.
 

\begin{figure}[ht]
		\centering	
    \includegraphics[width=0.32\textwidth]{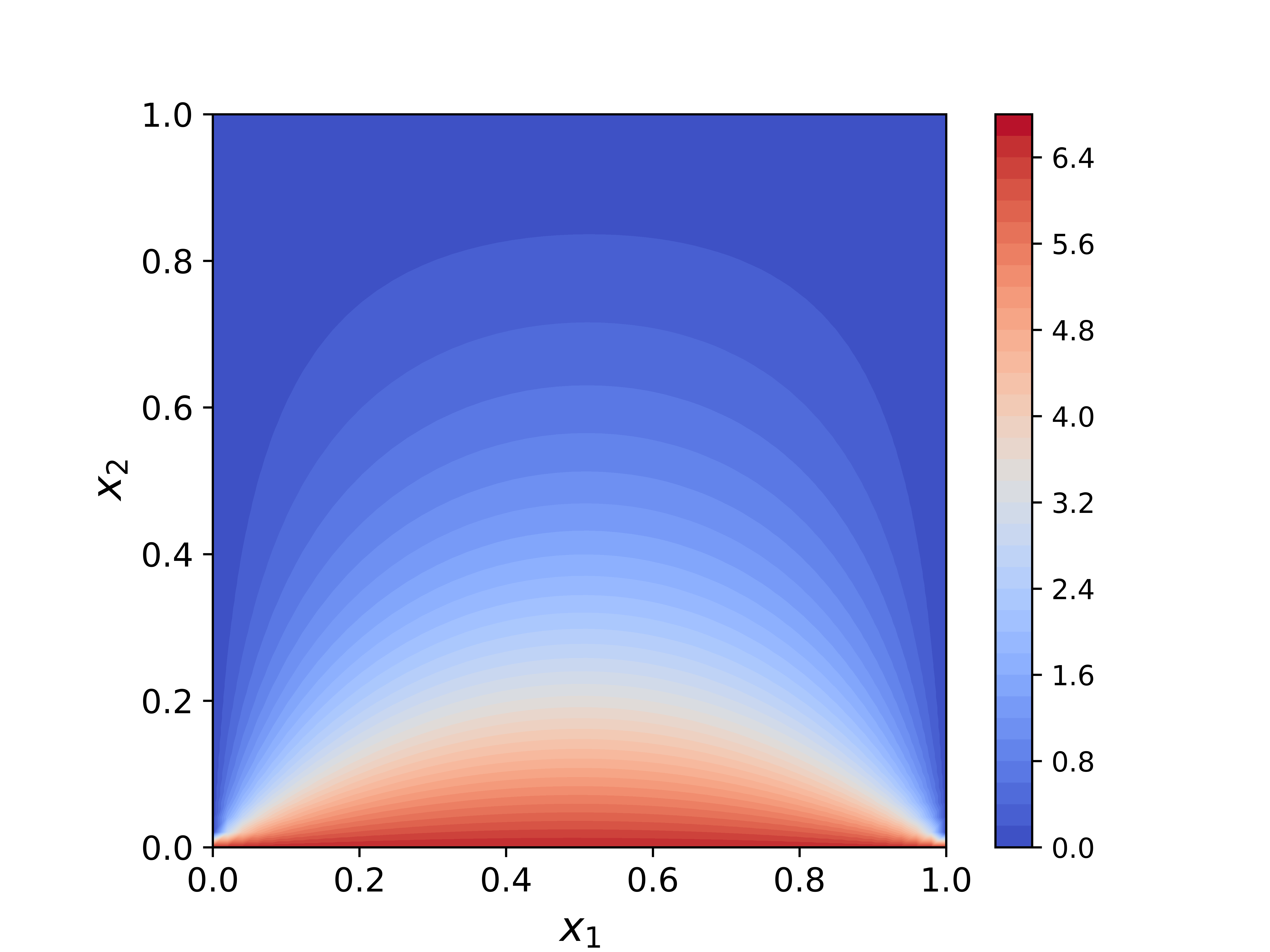}
    \hspace{0.05cm}
    \includegraphics[width=0.32\textwidth]{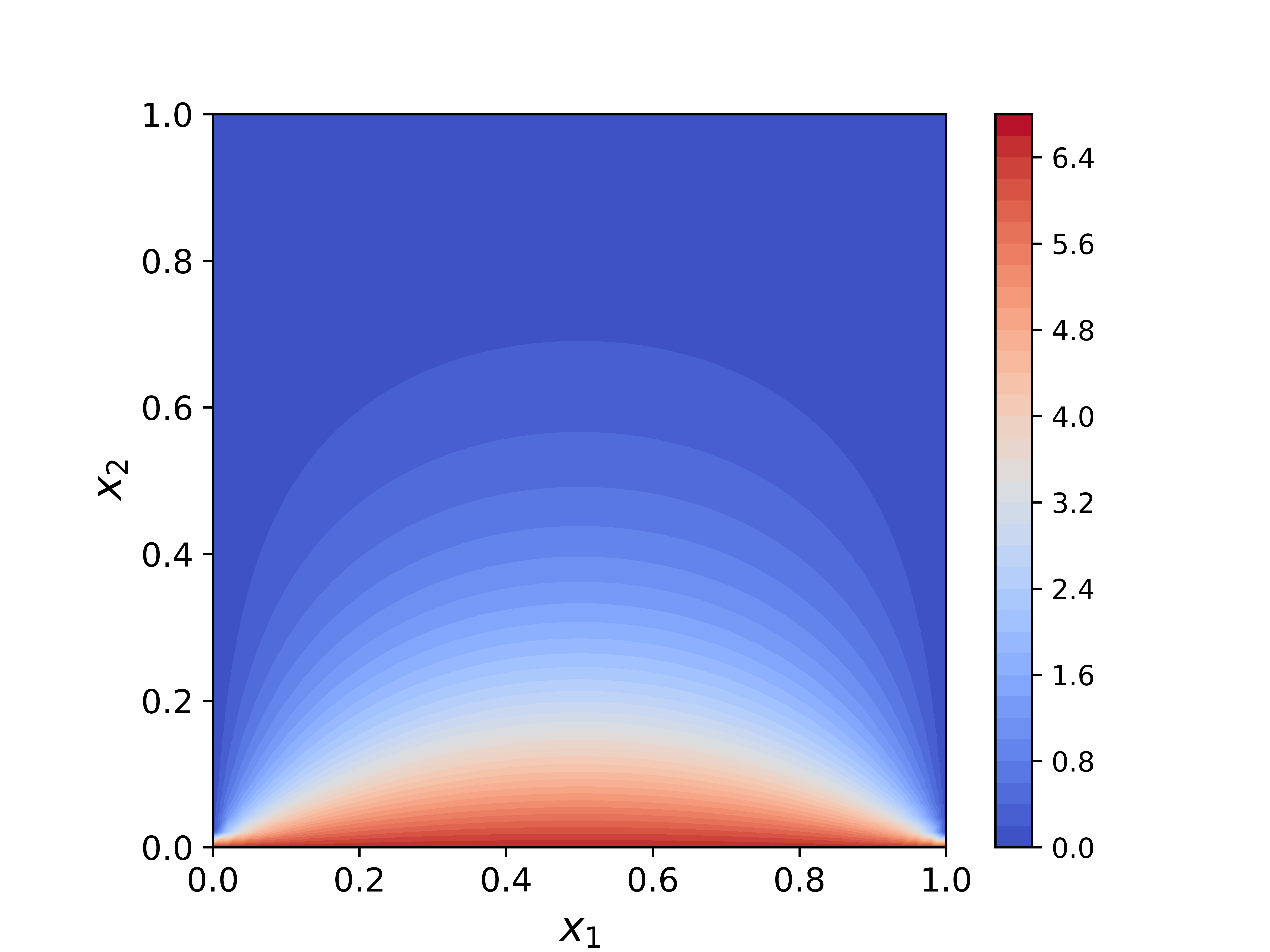}
    \hspace{0.05cm}
    \includegraphics[width=0.32\textwidth]{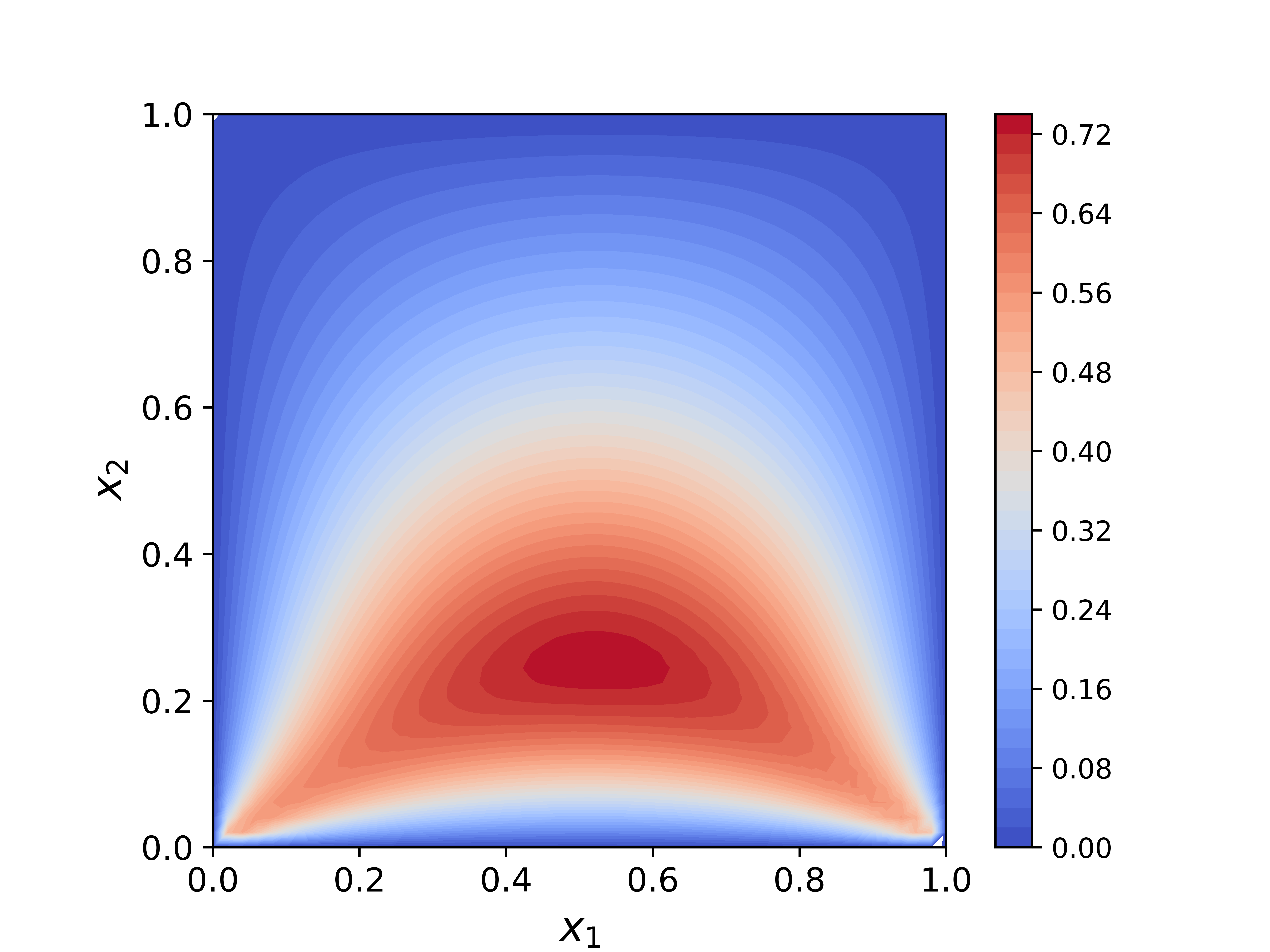}
		\caption{Macroscopic approximation of the solution based on scheme $2$ (precomputing)  at $T = 2$ and $M = 50$ with microscopic geometry $1$ (left), geometry $2$ (middle) and the difference between two solutions (right).}
		\label{macrosolution_scheme2_non_hom_diri}
\end{figure}

 Comparing the three plots in Figure \ref{macrosolution_scheme2_non_hom_diri}, we see a similar dispersive behavior of the concentration corresponding to both microscopic geometries discussed here.
However, the concentration profile corresponding to geometry $2$ (see the middle plot in Figure \ref{macrosolution_scheme2_non_hom_diri}) indicates a slower dispersion than that of geometry $1$.
This is an expected outcome due to the presence of two long rectangular horizontal obstacles in geometry $2$ that hinder the flow in the vertical direction.
To further investigate the influence of the microscopic geometry on macroscopic flow, we present the difference between microscopic solutions corresponding to two geometries in the right graph of Figure \ref{macrosolution_scheme2_non_hom_diri}.   

\section{Conclusion and outlook} \label{conclusion}
In this paper, we numerically study a two-scale system with nonlinear dispersion.  
We employ two decoupling strategies that make use of finite element methods to approximate weak solutions to the proposed two-scale system.
 Inspired by our earlier work \cite{raveendran2023strongly}, we first implement a Picard-type iterative scheme (referred to as scheme $1$), which decouples the system using an iteration. 
 Since the convergence of this scheme is known, we take this scheme as a baseline for comparison.
 We then constructed and implemented another scheme (named scheme $2$) which decouples the problem through time-stepping. 
Comparing the two schemes, scheme 2 is computationally more efficient and perhaps a more natural discretization of the problem.
However, implementing the precomputing strategy discussed in \cref{precomputing_strategy} makes both schemes viable in practice, saving considerable computing time for both schemes, but introducing interpolation error.  This interpolation error can be controlled by refining the parameter step size. 
We also expect that the numerical approximation strategies presented here can be utilized for other large classes of coupled two-scale systems (like those proposed in  \ME{\cite{olivares2021two}}, e.g.).
The mathematical analysis of Problem $(P)$ can be extended to include additional model components at each of the two spatial scales, provided the data is sufficiently smooth.

From a practical point of view, the simulation results are promising.
They suggest how the choice of the microscopic geometry can affect the macroscopic dispersion. 
Interestingly, through careful selection of parameters, localized fast and slow dispersion can be tailored to happen. 
What this means for concrete applications (e.g., penetration in rubber, oil extraction, etc.) is yet to be explored.

The well-posedness study of scheme $2$ and the corresponding FEM error analysis for both schemes $1$ and $2$ are worth investigating. 
We plan to do so at a later time, benefiting from preliminary studies done in \ME{\cite{nepal2023analysis}} where related results were proven.
We intend to investigate whether the parallel-in-time iterative scheme proposed in \cite{borregales2019partially} is applicable to our problem.

\appendix
\section{Appendix} \label[appendix]{appendix}
We describe here the necessary details that are needed for the computation of the vector field $B$ that arises explicitly in the formulation of 
Problem $(P)$ as the microscopic drift. Our attention is focused here on the following Stokes problem formulated in terms of $B$, namely:

 \begin{subequations}\label{Eq:Stokes_System}
 \begin{alignat}{2}
 \label{stoke1}- \mu \Delta B + \nabla p &= F(y) \quad & &\text{in}\;\; Y,\\
 \label{stoke2}\text{div}\, B &= 0 \;\; & &\text{in}\;\; Y,\\
\label{stoke3}  B &= 0 \;\; & &\text{on} \;\;\Gamma_N,\\
 \label{stoke4} y&\mapsto B(y) &\;&\text{is} \; Y\text{-periodic}, 
 \end{alignat}
  \end{subequations}
\noindent
To solve this Stokes problem, we pose \eqref{stoke1}-\eqref{stoke4} into a mixed variational form. 
 To do so,  we introduce the function space
 \begin{equation}
   {H}^{1, \Gamma_N}_{\#}(Y)  = \{ v \in H^1(Y)\;:\; v = 0 \;\; \text{on }\;\; \Gamma_N \;\;\text{and}\;\; v \;\text{is}\; Y \;\text{periodic}  \}.
\end{equation}
 We construct  the weak form of \eqref{stoke1}-\eqref{stoke4} by multiplying \eqref{stoke1}   with the test function  $v_1\in (H^{1, \Gamma_N}_{\#}(Y))^2$ and  \eqref{stoke2} with the test function $ v_2 \in H^1(Y)$ and  integrate the corresponding results over the domain $Y$.  We then add them up to get  the following weak formulation: Find the couple 
$$(B, p) \in (H^{1, \Gamma_N}_{\#}(Y))^2 \times H^1(Y)$$ 
such that the following identity holds:
\begin{align}
\label{stokeweak}\int_Y \nabla_y B(y)\cdot \nabla_y v_1(y) \di{y} & +\int_Y   p(y)  \nabla \cdot v_1(y) \di{y} +  \int_Y  \nabla \cdot  B(y) v_2 \di{y} = \int_Y F(y) v_1 \di{y} 
\end{align}
for all $(v_1, v_2)\in (H^{1, \Gamma_N}_{\#}(Y))^2 \times H^1(Y)$. 
\noindent
The computation of the drift $B$ is dealt with in FEniCS, where we implement \eqref{stokeweak}.  To ensure the stability of the resulting finite element approximations, we employ the Taylor-Hood elements \cite{taylor1973numerical} as they are provided in FEniCS. To this end,  second-order polynomials are used as the basis function for the velocity field and first-order polynomials are used for the pressure. 
 We refer the reader, for instance, to \cite{arnold1984stable} (and references citing this material) for more information on the use of stable finite elements for the computation of the Stokes flow.
\SN{\section{Appendix}\label{appendix2}
We now illustrate by means of an example how the nonlinear drift interaction affects the macroscopic dispersion. The parameters are the same as in Figure \ref{macrosolution_scheme2}, except that we now consider the nonlinear drift structure
$
G(u)=1 / (0.0001+\lvert 1-2u\rvert).
$
The numerical results at $T=2$, obtained using scheme 2, with and without precomputing, are presented in Figure \ref{nonlinear_G}. 
\begin{figure}[hbt!]
		\centering	
    \includegraphics[width=0.45\textwidth]{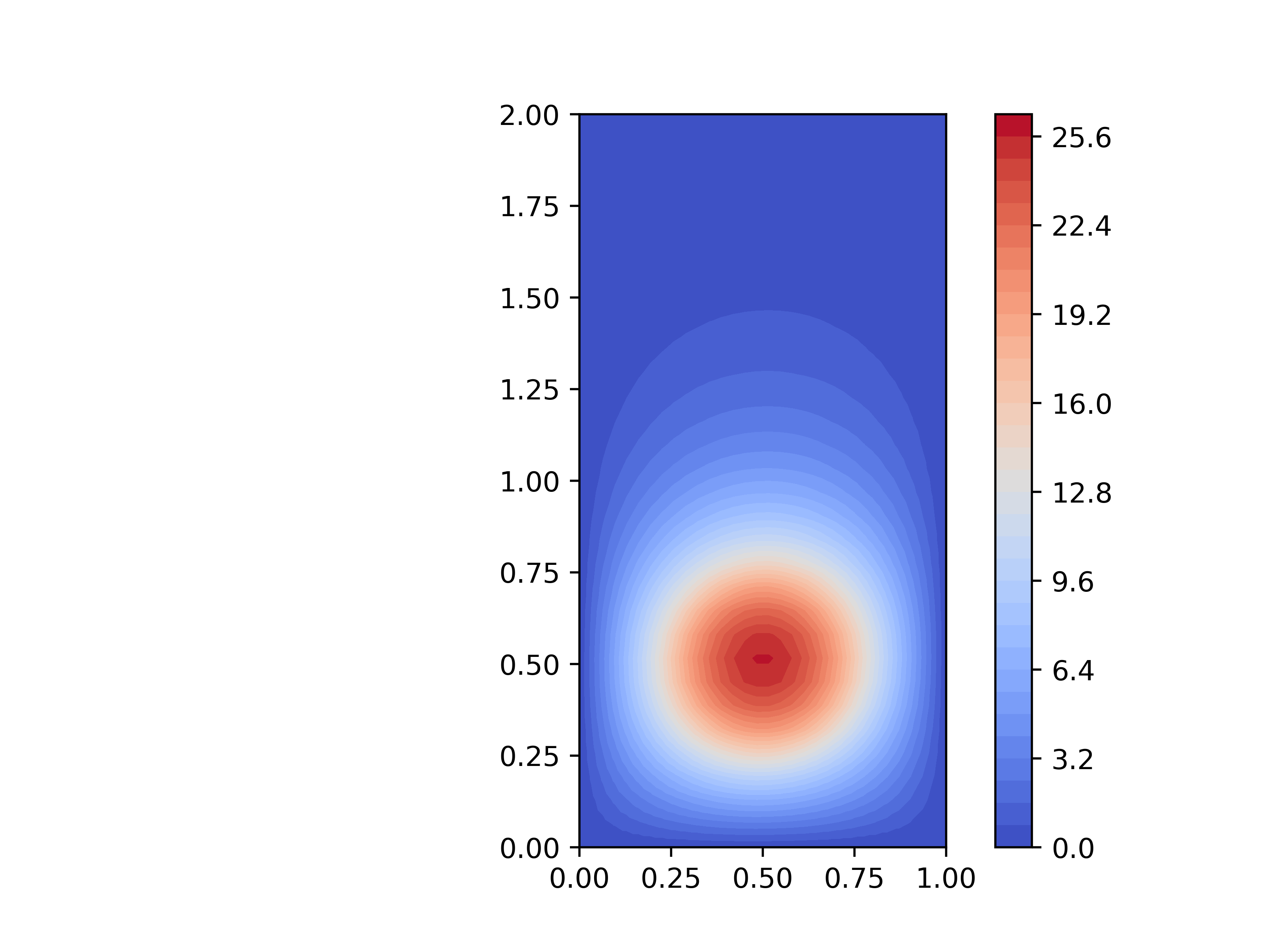}
    \hspace{-0.1cm}
\includegraphics[width=0.450\textwidth]{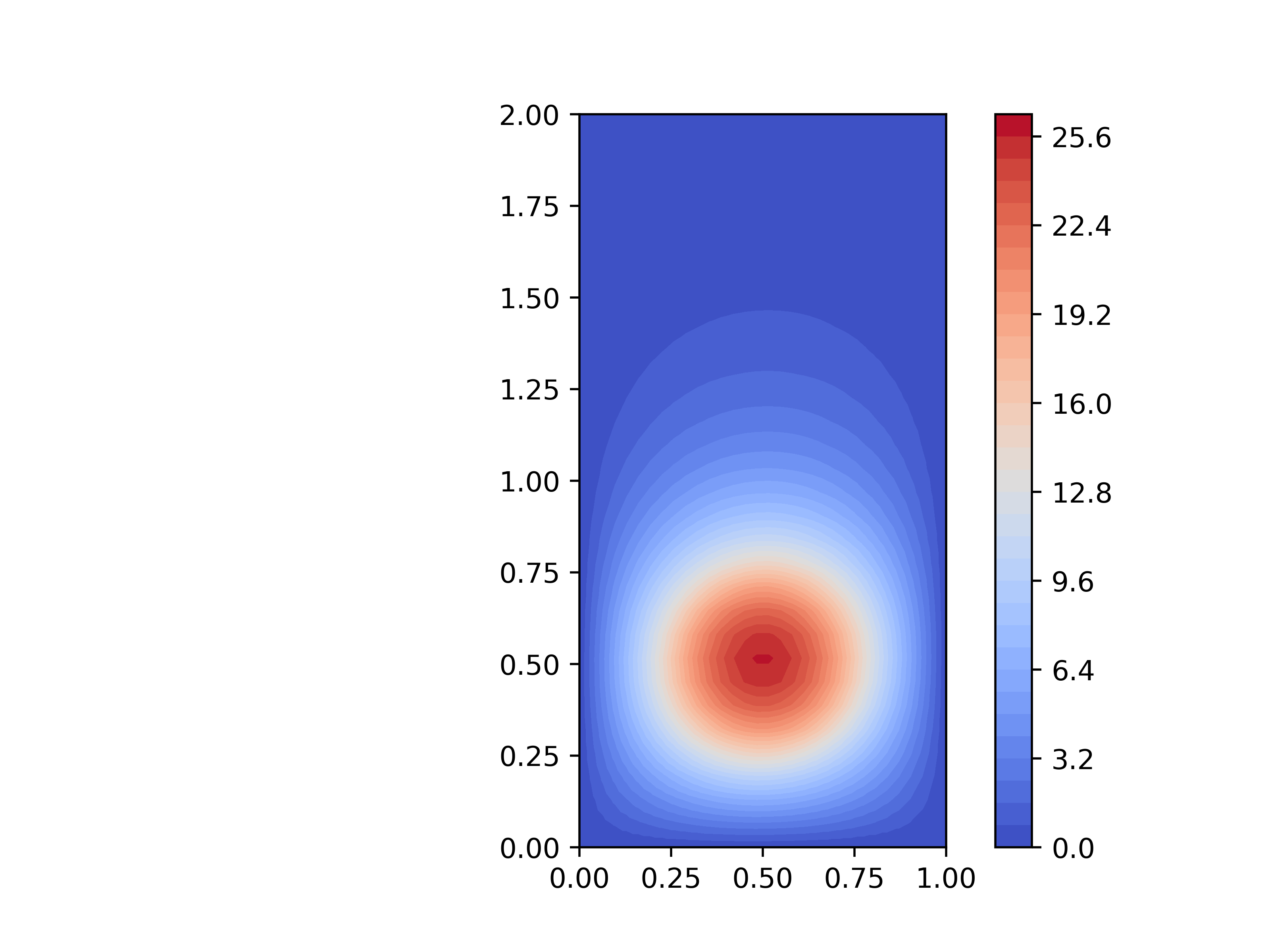}
		\caption{\SN{Concentration profile approximated via scheme 2 (left) and via scheme 2 with precomputing (right) with the nonlinear function $G=1/ (0.0001+\lvert 1-2u\rvert)$}.}
		\label{nonlinear_G}
\end{figure} 
Compared with the linear choice presented in Figure \ref{macrosolution_scheme2}, the concentration diffuses more slowly for this nonlinear interaction function. In particular, the concentration profile remains more localized around the source in the nonlinear case. Since the error behaviour and the computation time comparison are similar to those reported in Section \ref{erroranalysis}, we do not repeat the corresponding tables for this case.
For the corresponding results obtained using scheme 1 with the same nonlinear choice of $G(u)$, we refer the reader to \cite{raveendran2023strongly}. }
\section*{Acknowledgements} We thank T. Freudenberg (Bremen, Germany) for fruitful discussions during his visit to Karlstad.  
The work of V.R., S.N., and A.M. is partially supported by the Swedish Research Council's project ``{\em  Homogenization and dimension reduction of thin heterogeneous layers}" (grant nr. VR 2018-03648).
The research activity of M.E. is funded by the European Union’s Horizon 2022 research and innovation program under the Marie Skłodowska-Curie fellowship project {\em{MATT}} (project nr.~101061956).
R.L. and A.M. are grateful to Carl Tryggers Stiftelse for their financial support through the grant CTS 21:1656. 

\section*{Conflict of Interest }
The authors declare no potential conflict of interest.

 \bibliographystyle{abbrv}
	\bibliography{mybib}

\end{document}